\documentclass[times,sort&compress,3p]{elsarticle}

\journal{Submitted to Journal of Multivariate Analysis}

\usepackage{amsmath,amssymb,amsthm}
\usepackage[labelfont=bf]{caption} 

\usepackage{appendix} 
\usepackage{enumitem} 
\usepackage[bookmarksnumbered,colorlinks=true]{hyperref} 
\usepackage{centernot} 
\usepackage{dsfont} 
\usepackage{xurl} 
\usepackage{geometry} 
\theoremstyle{plain}
\newtheorem{theorem}{Theorem}
\newtheorem{proposition}{Proposition}
\newtheorem{lemma}{Lemma}
\newtheorem{corollary}{Corollary}

\theoremstyle{definition}

\newtheorem*{example}{Example}
\newtheorem{remark}{Remark}

\newcommand{\N}{\mathbb{N}}
\newcommand{\R}{\mathbb{R}}

\newcommand{\EE}{\mathsf{E}} 
\newcommand{\Var}{\mathsf{Var}} 
\newcommand{\bb}[1]{\boldsymbol{#1}}
\newcommand{\rd}{\mathrm{d}}
\newcommand{\tr}{\mathrm{tr}}
\newcommand{\etr}{\mathrm{etr}}
\newcommand{\vecc}{\mathrm{vec}}

\newcommand*{\bigcdot}{\mathbin{\raisebox{-0.5ex}{\scalebox{1.7}{\hspace{-0.15mm}$\cdot$}}}}

\allowdisplaybreaks

\begin{document}

\begin{frontmatter}

\title{Stein's method for the matrix normal distribution}

\author[a1]{Robert E.\ Gaunt}
\author[a2]{Fr\'ed\'eric Ouimet\corref{mycorrespondingauthor}}
\author[a3]{Donald Richards}

\address[a1]{Department of Mathematics, The University of Manchester, Manchester, M13 9PL, United Kingdom}
\address[a2]{D\'epartement de math\'ematiques et d'informatique, Universit\'e du Qu\'ebec \`a Trois-Rivi\`eres, Trois-Rivi\`eres (Qu\'ebec), G8Z 4M3, Canada}
\address[a3]{Department of Mathematics and Statistics, Penn State University, University Park, PA 16802, USA\vspace{-5mm}}

\cortext[mycorrespondingauthor]{Corresponding author. Email address: frederic.ouimet2@uqtr.ca}

\begin{abstract}
This work presents the first systematic development of Stein's method for matrix distributions. We establish the basic essential ingredients of Stein's method for matrix normal approximation: we derive an extended-generator-based Stein identity from a matrix Ornstein--Uhlenbeck diffusion with two-sided scales, provide an explicit semigroup representation for the solution of the Stein equation, and obtain regularity estimates for the solution. The new methodology is demonstrated in three examples: (i) smooth Wasserstein distance bounds to quantify the matrix central limit theorem (a didactic example), (ii) a Wasserstein distance bound for the matrix normal approximation of the centered matrix $T$ distribution, and (iii) a Stein's method-of-moments approach to estimating the row and column covariance factors of the matrix normal, yielding a flexible class of weighted flip--flop Stein estimators that generalize Dutilleul's classical flip--flop algorithm and naturally accommodate row/column importance weights, systematic missingness, and projection onto structured covariance families. The latter two examples are intrinsically matrix-valued and cannot be treated using naive vectorization.
\end{abstract}


\end{frontmatter}

\vspace{-3mm}
\section{Introduction}\label{sec:intro}

Stein's method is a powerful probabilistic technique for deriving explicit, quantitative bounds on distances between probability distributions. Introduced for normal approximation in the seminal paper \citep{s72}, the starting point of the method is the following \emph{Stein characterization} of the standard normal law. A real-valued random variable $X$ is equal in law to the standard normal distribution if and only if $\EE[f'(X)-Xf(X)]=0$ for all absolutely continuous functions $f:\R\to \R$ such that $\EE[|f'(Z)|]<\infty$ for $Z\sim \mathcal{N}(0,1)$. This characterization motivates the \emph{Stein equation} $\mathcal{A}f(x)=h(x)-\EE[h(Z)]$ for a test function $h:\R\to \R$, where the \emph{Stein operator} is given by $\mathcal{A}f(x)=f'(x)-xf(x)$. Letting $f_h$ denote the solution of the Stein equation, we have the transfer principle $\EE[h(W)]-\EE[h(Z)]=\EE[\mathcal{A}f_h(W)]$, which allows one to convert the problem of bounding the distance (with respect to some probability metric) between the distribution of $W$ and the standard normal law to bounding the expectation $\EE[\mathcal{A}f_h(W)]$ (over all $h$ in some function class), a task which is tractable in many settings on account of the fact that this is an expectation involving a single random variable. Stein's method for normal approximation is now very well developed, and we refer the reader to the monographs of \citet{MR2732624} and \citet{np12} for a thorough treatment.

A desirable feature of Stein's method is that the above procedure applies to general target distributions by selecting a suitable Stein operator to characterize the distribution. Indeed, the method has been adapted to many of the most important univariate probability distributions, including the Poisson \citep{c75}, exponential \citep{cfr11,pr11}, chi-square \citep{gpr17}, variance-gamma \citep{g14} and stable laws \citep{ah19book,x19}.

Stein's method has also been extended to certain multivariate distributions, such as the multivariate normal \citep{b90,g91} and multivariate stable laws \citep{ah19,stable24}. For multivariate normal approximation, Stein's method is now highly refined with the power of the theory exhibited by recent results quantifying the multivariate central limit theorem in the convex and Wasserstein distances \citep{MR4003566,bonis,MR3980309}. Moreover, a general theory of Stein's method for multivariate distributions is beginning to emerge \citep{mrrs23}. However, despite these advances in Stein's method for distributions of random vectors, a systematic framework for \emph{matrix} laws is not available in the present literature. This gap matters: vectorizing a random matrix by stacking its columns obscures the inherent algebraic structure, left--right symmetries, and intrinsic dependencies that motivate matrix models in the first place. For example, matrix $T$ laws arise naturally from normal--inverse-Wishart mixtures in matrix regression and MANOVA, and their geometry is genuinely matrix-specific; it is not faithfully captured by naive vectorization \citep[p.~140]{MR1738933}. Likewise, core inference tasks such as estimating the row and column covariance factors $(\Psi,\Sigma)$ in a centered matrix normal model benefit from identities and operators that respect the Kronecker structure $\Psi\otimes\Sigma$ rather than treat the underlying observations as long vectors.

Over the years, Stein's method has been used to study problems concerning matrix distributions. Existing ``matrix Stein'' ideas have largely targeted concentration inequalities \citep{MR3189061,MR2946459,doi:10.1561/2200000048,doi:10.1214/26-EJP1472} or distributional approximations for particular statistics of random matrices, such as the rank distribution \citep{fg15}. Stein's method and the Malliavin--Stein method have also been used for distributional approximation of matrix-valued laws themselves, which has typically been done by reducing the problem from one of matrix-variate approximation to one of random-vector approximation and exploiting the framework of Stein's method for multivariate normal approximation \citep{m22,nz22,tudor}.

The purpose of this paper is to provide the first systematic development of Stein's method for matrix distributions. We focus on the matrix normal distribution on account of its central role in probability and statistics, and the fact that an extensive literature has now emerged on Stein's method for normal and multivariate normal approximation. In this work, we lay the foundations for the matrix normal setting. We begin by applying the generator approach \citep{b90,g91} to obtain a Stein equation for the matrix normal distribution. Specialized to the centered matrix normal, the Stein operator is the extended generator of a matrix Ornstein--Uhlenbeck diffusion with two-sided scales (one-sided versions were studied by \citet{Bru1987,MR1132135,arXiv:1201.3256}), which yields an extended generator identity with the matrix normal as the unique stationary distribution. Via this approach, we are naturally led to an explicit semigroup representation for the solution of the Stein equation. We then proceed to establish regularity estimates for this solution, extending widely used bounds for partial derivatives thereof in the multivariate normal setting to the matrix normal setting. Our treatment stays in matrix coordinates, preserving row/column anisotropy and the Kronecker geometry.

In Section~\ref{sec:examples}, we provide three illustrative examples of our Stein framework for matrix normal approximation. We open in Section~\ref{sec:empirical.mean} with an explicit order $n^{-1/2}$ bound to quantify the matrix central limit theorem with respect to a smooth Wasserstein distance, with constants that make explicit the row- and column-wise anisotropy induced by the Kronecker covariance structure. In Section~\ref{sec:matrix.T}, we obtain a Wasserstein distance bound between the centered matrix $T$ distribution and its Gaussian counterpart. In Section~\ref{sec:parameter.estimation}, we combine our Stein characterization of the matrix normal distribution with Stein's method of moments \citep{MR4986908} to construct a flexible class of method-of-moments Stein estimators for the row and column scale matrices, based on matrix-valued test functions. This yields weighted flip--flop updates that naturally handle row/column importance weights, systematic missingness, and projections onto structured covariance families, and we recover the flip--flop maximum-likelihood estimating equations of \citet{doi:10.1080/00949659908811970} as a special case. This example highlights the utility of the matrix normal Stein framework for statistical tasks beyond classical distributional approximation. It also aligns with the expanding role of Stein’s method in areas such as goodness-of-fit testing \citep{survey23}, parameter estimation \citep{MR4986908}, and the relaxation of the Gaussian assumption in shrinkage estimation \citep{shrink}.

The paper is organized as follows. Section~\ref{sec:definitions} gathers the necessary preliminary definitions and notation. Section~\ref{sec:main.results} develops the Stein framework for the matrix normal distribution. Section~\ref{sec:examples} illustrates the theory with three examples. Section~\ref{sec:proofs} contains the proofs of all results, with some technical lemmas deferred to \ref{app:tech.lemmas}.

\section{Definitions and notation}\label{sec:definitions}

Throughout, $[d]=\{1,\ldots,d\}$ for $d\in\N \equiv \{1,2,\ldots\}$. Let $\mathcal{S}^d$, $\mathcal{S}_+^d$, and $\mathcal{S}_{++}^d$ denote, respectively, the sets of real symmetric, nonnegative definite, and positive definite $d\times d$ matrices. For $\nu,d\in\N$, let $\R^{\nu\times d}$ denote the space of real $\nu\times d$ matrices, equipped with the Frobenius inner product $\langle A, B \rangle_F = \tr(A^{\top} B)$, and the induced norm $\|A\|_F = \!\sqrt{\langle A,A\rangle_F}$. For any square matrix $A$, let $\tr(A)$ be its trace, $\etr(A)=\exp\{\tr(A)\}$, and $|A|$ its determinant. For $S\in\mathcal{S}_+^d$, the matrix $S^{1/2}$ denotes the symmetric square root and $\|S\|_2$ the spectral norm. If $B\subseteq\R^{\nu\times d}$ is open and $m\in\N$, let $C^{m}(B)$ be the class of real-valued functions $f:B\to\R$ that are $m$ times continuously differentiable on $B$ (all partial derivatives up to total order $m$ exist and are continuous), and let $C_b^{m}(B)$ be the subclass for which these derivatives are bounded on $B$. The symbols $\bb{0}_d$, $0_{\nu\times d}$, and $I_d$ denote the $d$-dimensional zero vector, the $\nu\times d$ zero matrix, and the $d\times d$ identity, respectively. The symbol $\rightsquigarrow$ denotes convergence in distribution.

Given any $\nu,d\in \N$, $M\in \R^{\nu\times d}$, $\Psi\in \mathcal{S}_{++}^{\nu}$, and $\Sigma\in \mathcal{S}_{++}^d$, the random matrix $\mathfrak{N}$ is said to have a matrix normal distribution, written $\mathfrak{N}\sim \mathcal{N}_{\nu\times d}(M, \Psi \otimes \Sigma)$, if $\vecc(\mathfrak{N}^{\top})\sim \mathcal{N}_{\nu d}(\vecc(M^{\top}), \Psi \otimes \Sigma)$. Here, $\vecc(\cdot)$ denotes the vectorization operator that stacks the columns of a matrix on top of each other, $\mathcal{N}_{\nu d}$ is the $(\nu d)$-dimensional multivariate normal distribution, and $\otimes$ denotes the Kronecker product. The corresponding density function with respect to the Lebesgue measure is given, for all $X \in \R^{\nu\times d}$, by
\[
\phi_{M,\Psi,\Sigma}(X) = (2\pi)^{-\nu d/2} |\Psi|^{-d/2} |\Sigma|^{-\nu/2} \etr\left\{-\frac{1}{2} \Sigma^{-1} (X - M)^{\top} \Psi^{-1} (X - M)\right\};
\]
see, e.g., \citet[Theorem~2.2.1]{MR1738933}.

For a matrix-variate Markov process $(\mathfrak{M}_t)_{t\geq 0}$ taking values in $\R^{\nu\times d}$, the transition semigroup of operators $(\mathcal{P}_t)_{t\geq 0}$ is defined, for every measurable function $f$ for which the expectation below is finite, by
\[
\mathcal{P}_t f(M) = \EE[f(\mathfrak{M}_t) \mid \mathfrak{M}_0 = M], \qquad t\geq 0.
\]
The corresponding infinitesimal generator of $(\mathfrak{M}_t)_{t\geq 0}$ is defined on its domain by
\begin{equation}\label{eq:generator}
\mathcal{A} f(M) = \lim_{s\downarrow 0} \frac{\mathcal{P}_s f(M) - f(M)}{s},
\end{equation}
provided that the limit exists. More generally, for a diffusion, we use the same notation for the extended generator: if $f$ is sufficiently smooth and there exists a measurable function $g$ such that $\smash{(f(\mathfrak{M}_t) - f(\mathfrak{M}_0)-\int_0^t g(\mathfrak{M}_s) \, \rd s)_{t\geq 0}}$ is a local martingale, then we write $\mathcal{A}f = g$. When $f$ belongs to the domain of the infinitesimal generator, the two notions agree. Hence, It\^o's formula identifies $\mathcal{A} f(M)$ with the drift term in the stochastic differential of $f(\mathfrak{M}_t)$. Such a characterization is central both to describing the dynamics of the Markov process $(\mathfrak{M}_t)_{t\geq 0}$ and to deriving Stein-type identities.

\section{Main results}\label{sec:main.results}

The matrix Ornstein--Uhlenbeck (OU) process featuring a diffusion coefficient with a one-sided scale was introduced by \citet{Bru1987,MR1132135} and is well known to have a matrix normal stationary limiting distribution \citep[Section~3.5]{arXiv:1201.3256}. To obtain stationary matrix normal distributions spanning the full range of possible covariance matrices, we consider a modified version of the process driven by a diffusion coefficient with two-sided scales. Specifically, consider the $\R^{\nu\times d}$-valued process $(\mathfrak{X}_t)_{t\geq 0}$ defined through the following stochastic differential equation (SDE):
\begin{equation}\label{eq:OU.process}
\begin{aligned}
\rd \mathfrak{X}_t
&= -\mathfrak{X}_t \rd t + \!\sqrt{2} \, \Psi^{1/2} \rd \mathfrak{B}_t \Sigma^{1/2}, \qquad \mathfrak{X}_0 = X_0, \\
\end{aligned}
\end{equation}
where $X_0\in \R^{\nu\times d}$, $\Psi\in \mathcal{S}_{++}^{\nu}$, $\Sigma\in \mathcal{S}_{++}^d$ are constant matrices, and $(\mathfrak{B}_t)_{t\geq 0}$ is a $\nu\times d$ matrix of independent standard Brownian motions.

The explicit expression for the extended generator of $(\mathfrak{X}_t)_{t\geq 0}$, denoted $\mathcal{A}^{\mathrm{OU}}$, is derived in Proposition~\ref{prop:OU.process.generator} below.

\begin{proposition}[Extended generator]\label{prop:OU.process.generator}
For any $f\in C^2(\R^{\nu\times d})$, we have
\begin{equation}\label{eq:OU.process.generator}
\mathcal{A}^{\mathrm{OU}}f(X) = -\tr\{X^{\top} \nabla f(X)\} + \tr\{\Sigma \nabla^{\top} \Psi \nabla f(X)\},
\end{equation}
where $\nabla = (\partial / \partial X_{ij})_{1\leq i \leq \nu, 1 \leq j \leq d}$ denotes the $\nu\times d$ matrix of first-order partial derivatives.
\end{proposition}

In the special case $\Psi = I_{\nu}$, Theorems~3.49--3.50 of \citet{arXiv:1201.3256} show that the distribution of $\mathfrak{X}_t$ given $\mathfrak{X}_0 = X$ and the stationary limiting distribution of $(\mathfrak{X}_t)_{t\geq 0}$ are given, respectively, by
\[
\mathfrak{X}_t \mid \{\mathfrak{X}_0 = X\} \sim \mathcal{N}_{\nu\times d}(e^{-t} X, (1 - e^{-2t}) (I_{\nu} \otimes \Sigma)), \qquad
\mathfrak{X}_{\infty} \sim \mathcal{N}_{\nu\times d}(0_{\nu\times d}, I_{\nu} \otimes \Sigma).
\]
The next proposition generalizes this result for general $\Psi\in \mathcal{S}_{++}^{\nu}$.

\begin{proposition}\label{prop:OU.process.distribution}
For the matrix Ornstein--Uhlenbeck process defined in \eqref{eq:OU.process}, we have
\[
\mathfrak{X}_t \mid \{\mathfrak{X}_0 = X\} \sim \mathcal{N}_{\nu\times d}(e^{-t} X, (1 - e^{-2t}) (\Psi \otimes \Sigma)), \qquad t > 0, \qquad
\mathfrak{X}_{\infty} \sim \mathcal{N}_{\nu\times d}(0_{\nu\times d}, \Psi \otimes \Sigma).
\]
In particular, let $(\mathcal{P}^{\mathrm{OU}}_t)_{t\ge0}$ be the transition semigroup with kernel $P_t(X,\rd Y)$, so that
\[
(\mathcal{P}^{\mathrm{OU}}_t h)(X) = \int_{\R^{\nu\times d}} h(Y)\,P_t(X,\rd Y)=\EE\big[h(\mathfrak{X}_t)\mid \mathfrak{X}_0=X\big].
\]
For $\mathfrak{Z}\sim \mathcal{N}_{\nu\times d}(0_{\nu\times d}, I_{\nu} \otimes I_d)$, the above shows
\[
\mathcal{P}_t^{\mathrm{OU}} h(X) = \EE\big[h(e^{-t} X + \!\sqrt{1 - e^{-2t}} \, \Psi^{1/2} \mathfrak{Z} \Sigma^{1/2})\big].
\]
Moreover, given a probability measure $\mu$ on $\R^{\nu\times d}$, the pushed-forward measure $\mu \, \mathcal{P}^{\mathrm{OU}}_t$ is defined by
\[
(\mu \, \mathcal{P}^{\mathrm{OU}}_t)(A) = \int_{\R^{\nu\times d}} P_t(X,A)\,\mu(\rd X), \qquad A\subseteq \R^{\nu\times d} \text{ Borel}.
\]
Hence, for $\gamma = \mathcal{N}_{\nu\times d}(0_{\nu\times d}, \Psi \otimes \Sigma)$, we have the invariance
\begin{equation}\label{eq:invariance}
\gamma \, \mathcal{P}^{\mathrm{OU}}_t = \gamma.
\end{equation}
\end{proposition}

This leads to the following Stein characterization of the matrix normal distribution.

\begin{corollary}[Stein characterization]\label{cor:Stein.normal.step.1}
Let $\Psi\in \mathcal{S}_{++}^{\nu}$ and $\Sigma\in \mathcal{S}_{++}^d$ be given. Then
\[
\mathfrak{X}\sim \mathcal{N}_{\nu\times d}(0_{\nu\times d}, \Psi \otimes \Sigma) \qquad \Leftrightarrow \qquad \EE\big[\mathcal{A}^{\mathrm{OU}}f(\mathfrak{X})\big] = 0 ~~\forall f\in C_{\mathcal{A}^{\mathrm{OU}}}^2(\R^{\nu\times d}),
\]
where
\[
\begin{aligned}
C_{\mathcal{A}^{\mathrm{OU}}}^2(\R^{\nu\times d})
&= \Big\{f\in C^2(\R^{\nu\times d}) \, : \, \EE[|\tr\{\mathfrak{X}^{\top} \nabla f(\mathfrak{X})\}|] < \infty, ~\EE[|\tr\{\Sigma \nabla^{\top} \Psi \nabla f(\mathfrak{X})\}|] < \infty \Big\}.
\end{aligned}
\]
\end{corollary}

For any $\alpha\in(0,1]$ and $h: \R^{\nu\times d}\to \R$, define the $\alpha$-H\"older seminorm
\begin{equation}\label{eq:seminorm}
[h]_{\alpha} = \sup_{X\neq Y} \frac{|h(X) - h(Y)|}{\|X - Y\|_F^{\alpha}},
\end{equation}
and the space of $\alpha$-H\"older continuous functions on $\R^{\nu\times d}$:
\[
C^{0,\alpha}(\R^{\nu\times d}) = \{h:\R^{\nu\times d} \to \R \mid [h]_{\alpha} < \infty\}.
\]
The class $C^{0,1}(\R^{\nu\times d})$ corresponds to the space of Lipschitz continuous functions on $\R^{\nu\times d}$ with minimum Lipschitz constant $[h]_1<\infty$. For $p\in\N\cup\{0\}$, we will let $\smash{\mathrm{Lip}_{p}(\R^{\nu\times d})}$ denote the class of functions $h$ on $\R^{\nu\times d}$ whose partial derivatives up to order $p$ exist (we use the convention that the zeroth-order partial derivative of a function is the function itself) and whose partial derivatives of order $p$ are in the class $C^{0,1}(\R^{\nu\times d})$. For $h\in\smash{\mathrm{Lip}_{p}(\R^{\nu\times d})}$, the minimum Lipschitz constant of $\smash{(\prod_{\ell=1}^{p} \nabla_{i_{\ell} j_{\ell}}) h}$ is $\smash{\mathrm{ess\,sup}_{X\in \R^{\nu\times d}} \|\nabla \, (\prod_{\ell=1}^{p} \nabla_{i_{\ell} j_{\ell}}) h\|_F}$. In particular, the partial derivatives of order $p+1$ in $\smash{\|(\prod_{\ell=1}^{p+1} \nabla_{i_{\ell} j_{\ell}}) h\|_{\infty}}$, and hence the order-$m$ derivatives of $h$ in Theorem~\ref{thm:smoothness.estimates} when $h\in\mathrm{Lip}_{m-1}(\R^{\nu\times d})$, are understood to exist almost everywhere and the corresponding $\|\cdot\|_\infty$ norms are interpreted as scalar essential supremum norms.

Theorem~\ref{thm:Stein.solutions} below provides an explicit solution $f_h: \R^{\nu\times d}\to \R$ to the matrix normal Stein equation,
\begin{equation}\label{eq:Stein.equation.normal}
\mathcal{A}^{\mathrm{OU}} f_h(X) = h(X) - \EE[h(\mathfrak{X}_{\infty})],
\end{equation}
for test functions $h$ belonging either to $C^{0,\alpha}(\R^{\nu\times d})$ for some $\alpha\in(0,1]$, or to $\smash{\mathrm{Lip}_{p}(\R^{\nu\times d})}$ for some $p\in\N\cup\{0\}$. For $\alpha$-H\"older test functions, it also provides a pointwise bound on the solution. Existence of the solution to the multivariate normal Stein equation under $\alpha$-H\"older regularity was established by \citet[Proposition 2.1]{gms}.

\begin{theorem}[Solutions of the matrix normal Stein equation]\label{thm:Stein.solutions}
Let $(\mathfrak{X}_t)_{t\geq 0}$ be the matrix Ornstein--Uhlenbeck process in \eqref{eq:OU.process} with transition semigroup $(\mathcal{P}^{\mathrm{OU}}_t)_{t\geq 0}$, extended generator $\mathcal{A}^{\mathrm{OU}}$ given by \eqref{eq:OU.process.generator}, and stationary limiting distribution $\gamma=\mathcal{N}_{\nu\times d}(0_{\nu\times d},\Psi\otimes\Sigma)$ from Proposition~\ref{prop:OU.process.distribution}. For every test function $h$ belonging either to $C^{0,\alpha}(\R^{\nu\times d})$ for some $\alpha\in(0,1]$, or to $\smash{\mathrm{Lip}_{p}(\R^{\nu\times d})}$ for some $p\in\N\cup\{0\}$, the function
\begin{equation}\label{eq:fh.def.OU.alpha}
f_h(X) = - \int_0^{\infty}\Big\{\mathcal{P}^{\mathrm{OU}}_t h(X) - \EE[h(\mathfrak{X}_{\infty})]\Big\}\rd t, \qquad X\in\R^{\nu\times d},
\end{equation}
is well defined pointwise and solves the matrix normal Stein equation \eqref{eq:Stein.equation.normal} in the pointwise semigroup-generator sense, i.e., for every $X\in\R^{\nu\times d}$, $\lim_{s\downarrow 0} \{(\mathcal{P}_s^{\mathrm{OU}} f_h(X)) - f_h(X)\}/s = h(X) - \EE[h(\mathfrak{X}_\infty)]$. Moreover, if $h\in C^{0,\alpha}(\R^{\nu\times d})$ for some $\alpha\in(0,1]$, then, letting $d_{\mathrm{HK},\alpha}$ denote the $\alpha$-H\"older--Kantorovich distance on $\R^{\nu\times d}$ induced by the $\alpha$-H\"older seminorm \eqref{eq:seminorm}, that is,
\begin{equation}\label{dha}
d_{\mathrm{HK},\alpha}(\mu,\eta) = \sup\left\{\Big|\int g \, \rd \mu - \int g \, \rd \eta\Big|: g\in C^{0,\alpha}(\R^{\nu\times d}), ~[g]_{\alpha} \leq 1\right\},
\end{equation}
we have
\begin{equation}\label{eq:fh.bound.OU.alpha}
|f_h(X)|
\leq \alpha^{-1} [h]_{\alpha} d_{\mathrm{HK},\alpha}(\delta_{X}, \gamma)
\leq \alpha^{-1} [h]_{\alpha} \Big\{\EE\big[\|\mathfrak{X}_{\infty}\|_F^{\alpha}\big] + \|X\|_F^{\alpha}\Big\},
\end{equation}
where $\delta_X$ is the point mass at $X$ and $\mathfrak{X}_{\infty}\sim \gamma$.
\end{theorem}

Next, we generalize to the matrix setting bounds on the partial derivatives of the solutions of the multivariate normal Stein equation.

\begin{theorem}[Regularity of the solutions of the matrix normal Stein equation]\label{thm:smoothness.estimates}
Let $\Psi\in \mathcal{S}_{++}^{\nu}$ and $\Sigma\in \mathcal{S}_{++}^d$ be given, and let $\mathfrak{Z}\sim \mathcal{N}_{\nu\times d}(0_{\nu\times d}, I_{\nu} \otimes I_d)$. If $h\in \mathrm{Lip}_{m-1}(\R^{\nu\times d})$ for some $m\in\N$, then
\begin{equation}\label{eq:Stein.bound.1}
\left\|\left(\prod_{\ell=1}^m \nabla_{i_{\ell} j_{\ell}}\right) f_h\right\|_{\infty}
\leq ~\frac{1}{m} \left\|\left(\prod_{\ell=1}^m \nabla_{i_{\ell} j_{\ell}}\right) h\right\|_{\infty}.
\end{equation}
Instead, if $h\in C^{0,\alpha}(\R^{\nu\times d})$ is bounded for some $\alpha\in(0,1]$, then
\begin{equation}\label{eq:Stein.bound.2}
\left\|\nabla_{ij} f_h\right\|_{\infty}
\leq \!\sqrt{\dfrac{\pi}{2}} \!\sqrt{(\Psi^{-1})_{ii} (\Sigma^{-1})_{\!j\hspace{-0.3mm}j}} ~\|h - \EE[h(\Psi^{1/2}\mathfrak{Z}\Sigma^{1/2})]\|_{\infty},
\end{equation}
and if $h\in \mathrm{Lip}_{m-2}(\R^{\nu\times d})$ with $m\in\N\setminus\{1\}$, then
\begin{equation}\label{eq:Stein.bound.3}
\left\|\left(\prod_{\ell=1}^m \nabla_{i_{\ell} j_{\ell}}\right) f_h\right\|_{\infty}
\leq \frac{\Gamma(m/2)}{\!\sqrt{2} \, \Gamma(m/2+1/2)} \min\limits_{1\leq k \leq m} \left\{\!\sqrt{(\Psi^{-1})_{i_k i_k} (\Sigma^{-1})_{j_k j_k}} \left\|\left(\prod_{\substack{\ell=1 \\ \ell\neq k}}^m \nabla_{i_{\ell} j_{\ell}}\right) h\right\|_{\infty}\right\}.
\end{equation}
Moreover, for $k\in\N$ and any $g:\R^{\nu\times d}\to\R$, define the supremum norm (in $X$) of the operator norm of the $k$th Fr\'echet derivative as follows:
\[
\mathcal{M}_k(g) = \sup_{X\in\R^{\nu\times d}}\ \sup_{\|U_1\|_F = \dots = \|U_k\|_F = 1} \big|D^k g(X)[U_1,\ldots,U_k]\big|.
\]
If $h\in C_b^m(\R^{\nu\times d})$ for some $m\in\N$, then
\begin{equation}\label{eq:Stein.bound.4}
\mathcal{M}_m(f_h)\leq \frac{1}{m}\mathcal{M}_m(h).
\end{equation}
If $h\in C^{0,\alpha}(\R^{\nu\times d})$ is bounded for some $\alpha\in(0,1]$, then
\begin{equation}\label{eq:Stein.bound.5}
\mathcal{M}_1(f_h)\leq \!\sqrt{\dfrac{\pi}{2}} \|\Psi^{-1/2}\|_2 \|\Sigma^{-1/2}\|_2 \|h - \EE[h(\Psi^{1/2}\mathfrak{Z}\Sigma^{1/2})]\|_{\infty},
\end{equation}
and if $h\in C_b^{m-1}(\R^{\nu\times d})$ with $m\in\N\setminus\{1\}$, then
\begin{equation}\label{eq:Stein.bound.6}
\mathcal{M}_m(f_h)\leq \frac{\Gamma(m/2)}{\!\sqrt{2} \, \Gamma(m/2 + 1/2)} \|\Psi^{-1/2}\|_2 \|\Sigma^{-1/2}\|_2 \, \mathcal{M}_{m-1}(h).
\end{equation}
\end{theorem}

\begin{remark}
Inequality \eqref{eq:Stein.bound.1} generalizes an analogous bound for the solution of the multivariate normal Stein equation which is given in \cite{b90,gr96,rr09}. Inequalities \eqref{eq:Stein.bound.2}, \eqref{eq:Stein.bound.3}, \eqref{eq:Stein.bound.5} and \eqref{eq:Stein.bound.6} generalize bounds from \citet[Proposition~2.1]{MR3463084}. Inequality \eqref{eq:Stein.bound.4} generalizes the bound given in \citet[Lemma~2, part~1]{meckes}.
\end{remark}

\section{Examples}\label{sec:examples}

This section provides three examples of the matrix normal Stein framework. The first two examples concern distributional approximation, while our third example concerns parameter estimation. The results of Example~2 (Section~\ref{sec:matrix.T}) and Example~3 (Section~\ref{sec:parameter.estimation}) cannot be obtained by naive vectorization. In order to state our bounds, we introduce the probability metrics that we will work with. Let $\mu$ and $\eta$ be the probability measures of the $\R^{\nu\times d}$-valued random matrices $\mathfrak{X}$ and $\mathfrak{Y}$. Suppose that $\EE[\|\mathfrak{X}\|_F]<\infty$ and $\EE[\|\mathfrak{Y}\|_F]<\infty$. Then the Wasserstein distance between $\mu$ and $\eta$ is given by $d_{\mathrm{W}}(\mu,\eta)=d_{\mathrm{HK},1}(\mu,\eta)$, where $d_{\mathrm{HK},\alpha}(\mu,\eta)$ is defined in \eqref{dha}. Now, let $p\in\N$ and suppose that $\EE[\|\mathfrak{X}\|_F^p]<\infty$ and $\EE[\|\mathfrak{Y}\|_F^p]<\infty$. For $p\in\N$, we will also work with smooth Wasserstein distances given by
\[
d_{p}(\mu,\eta)=\sup_{h\in\mathcal{H}_p}\big|\EE[h(\mathfrak{X})] - \EE[h(\mathfrak{Y})]\big|, \vspace{-5mm}
\]
where
\[
\mathcal{H}_p
= \Big\{h:\R^{\nu\times d}\to \R \, : \, \text{$h$ is $(p-1)$-times differentiable and $\textstyle[(\prod_{\ell=1}^{p-1} \nabla_{i_{\ell} j_{\ell}}) h]_1\leq1$ for all $1\leq i_{\ell}\leq \nu$, $1\leq j_{\ell}\leq d$} \Big\}.
\]
Note that $d_{1}(\mu,\eta)=d_{\mathrm{W}}(\mu,\eta)$.

\subsection{Example 1: Comparison between the standardized sample mean of iid matrix-valued observations and the matrix normal distribution with the same covariance}\label{sec:empirical.mean}

We open by deriving smooth Wasserstein distance bounds to quantify the matrix normal approximation of a standardized sum of iid random matrices. In the multivariate case, this problem has provided a canonical first example to illustrate the use of coupling techniques within Stein's method for multivariate normal approximation \citep{gr96,gr05,cm08,rr09} in which the distributional approximation is quantified in (the equivalent of) smooth Wasserstein distances. Since these works, Stein's method for multivariate normal approximation has reached a rather sophisticated level, with the power of the theory exhibited by recent results quantifying the multivariate central limit theorem in the convex and Wasserstein distances \citep{MR4003566,bonis,MR3980309}. Our bounds given in the following proposition can be viewed as matrix normal analogs of the bounds of \citet{gr96,gr05,cm08,rr09}, and provide a simple exposition of how our Stein framework for matrix normal approximation works in conjunction with local approach couplings.

\begin{proposition}\label{prop:generalization.Gaunt.2013.Example.2.6}
Let $\mathfrak{X},\mathfrak{X}_1,\mathfrak{X}_2,\ldots,\mathfrak{X}_n$ be a collection of iid random matrices of size $\nu\times d$ such that $\EE[\mathfrak{X}] = 0_{\nu\times d}$, $\smash{\EE[\vecc(\mathfrak{X}^{\top}) \vecc(\mathfrak{X}^{\top})^{\top}] = \Psi \otimes \Sigma}$, and $\EE[|\mathfrak{X}_{ij}|^3] < \infty$ for all $(i,j)\in [\nu] \times [d]$. Let $\smash{\mathfrak{S} = n^{-1/2} \sum_{k=1}^n \mathfrak{X}_k}$ and $\mathfrak{Z}\sim\mathcal{N}_{\nu\times d}(0_{\nu\times d},\Psi\otimes\Sigma)$. Denote the probability measures of $\mathfrak{S}$ and $\mathfrak{Z}$ by $\mu_n$ and $\gamma$, respectively. Then
\begin{align}
d_3(\mu_n,\gamma)
&\leq \frac{1}{3\!\sqrt{n}}\sum_{i_1,i_2,i_3=1}^{\nu} \sum_{j_1,j_2,j_3=1}^d \left\{\frac{1}{2} \EE[|\mathfrak{X}_{i_1j_1} \mathfrak{X}_{i_2j_2} \mathfrak{X}_{i_3j_3}|] + |\Psi_{i_1i_2} \Sigma_{j_1j_2}| \EE[|\mathfrak{X}_{i_3j_3}|]\right\}, \label{cltbd1} \\
d_2(\mu_n,\gamma)
&\leq \frac{\!\sqrt{2\pi}}{4\!\sqrt{n}} \, \max_{i\in [\nu], j\in [d]} \!\sqrt{(\Psi^{-1})_{ii} (\Sigma^{-1})_{\!j\hspace{-0.3mm}j}} \sum_{i_1,i_2,i_3=1}^{\nu} \sum_{j_1,j_2,j_3=1}^d \left\{\frac{1}{2} \EE[|\mathfrak{X}_{i_1j_1} \mathfrak{X}_{i_2j_2} \mathfrak{X}_{i_3j_3}|] + |\Psi_{i_1i_2} \Sigma_{j_1j_2}| \EE[|\mathfrak{X}_{i_3j_3}|]\right\}. \label{cltbd2}
\end{align}
\end{proposition}

\begin{remark}\label{rem:empirical.mean.smooth.distances}
Proposition~\ref{prop:generalization.Gaunt.2013.Example.2.6} is stated only for smooth Wasserstein distances (not $d_{\mathrm{W}} = d_1$) because its proof uses a third-order Taylor expansion of the solution $f_h$ of the matrix normal Stein equation, and hence requires uniform control of third-order partial derivatives of $f_h$. Inequality (\ref{eq:Stein.bound.1}) of Theorem~\ref{thm:smoothness.estimates} gives this control directly for $h\in\mathcal{H}_3$, yielding \eqref{cltbd1}, while inequality \eqref{eq:Stein.bound.3} controls third-order partial derivatives of $f_h$ for $h\in\mathcal{H}_2$, yielding \eqref{cltbd2}. However, when $\nu\geq2$ or $d\geq2$, for general Lipschitz test functions $h:\R^{\nu\times d}\to \R$, it is not possible to obtain uniform bounds on the third-order partial derivatives of $f_h$ in terms of the Lipschitz constant of $h$ alone. This can be seen from Remark~2 of \citet{raic04} (see also Proposition 2.4 of \citet{gms}), which gives a Lipschitz test function $h:\R^2\to\R$ for which the mixed second-order partial derivative of the solution of the corresponding multivariate normal Stein equation is not Lipschitz. In particular, the Hessian of $f_h$ need not be Lipschitz, so uniform third-order derivative bounds cannot hold in general.

The theory of Stein's method for multivariate normal approximation has now developed to the point where this difficulty can be bypassed: researchers have established Wasserstein-type bounds for the multivariate central limit theorem with optimal and near-optimal rates of order $n^{-1/2}$ and $\log(n)/\!\sqrt{n}$, respectively \citep{bonis,MR3980309,gms}. A natural direction for future research is to further develop Stein's method for matrix normal approximation so that it can be used to derive optimal-order Wasserstein bounds for the matrix normal central limit theorem.
\end{remark}

\begin{remark}
The matrix framework is not necessary for Proposition~\ref{prop:generalization.Gaunt.2013.Example.2.6}; this example is included mainly as a benchmark to familiarize the reader with the matrix Stein setup. Indeed, under the identification $\R^{\nu\times d}\cong\R^{\nu d}$ given by the mapping $X\mapsto \vecc(X^{\top})$, set
\[
\bb{X} = \vecc(\mathfrak{X}^{\top}),\qquad \bb{X}_k=\vecc(\mathfrak{X}_k^{\top}), \qquad
\bb{S} = \vecc(\mathfrak{S}^{\top})=\frac{1}{\!\sqrt{n}}\sum_{k=1}^n \bb{X}_k,
\qquad \Omega = \Psi\otimes\Sigma,
\]
so that $\Var(\bb{X}) = \Omega$ and $\bb{Z} = \vecc(\mathfrak{Z}^{\top})\sim\mathcal{N}_{\nu d}(\bb{0}_{\nu d},\Omega)$. Applying the standard leave-one-out Taylor expansion for the multivariate normal Stein equation, together with the derivative bounds of \citet[Proposition~2.1]{MR3463084}, gives
\[
\begin{aligned}
d_3\big(\mathrm{Law}(\bb{S}),\mathrm{Law}(\bb{Z})\big)
&\leq \frac{1}{3\!\sqrt{n}}\sum_{a,b,c=1}^{\nu d}\left\{\frac{1}{2} \, \EE\big[|\bb{X}_a\bb{X}_b\bb{X}_c|\big] + |\Omega_{ab}|\,\EE\big[|\bb{X}_c|\big]\right\}, \\
d_2\big(\mathrm{Law}(\bb{S}),\mathrm{Law}(\bb{Z})\big)
&\leq \frac{\!\sqrt{2\pi}}{4\!\sqrt{n}} \, \max_{a\in[\nu d]}\!\sqrt{(\Omega^{-1})_{aa}} \, \sum_{a,b,c=1}^{\nu d}\left\{\frac{1}{2} \, \EE\big[|\bb{X}_a\bb{X}_b\bb{X}_c|\big] + |\Omega_{ab}|\,\EE\big[|\bb{X}_c|\big]\right\}.
\end{aligned}
\]
Reindexing $(a,b,c)=((i_1,i_2,i_3) - (1,1,1)) \, d+(j_1,j_2,j_3)$ and using $\Omega_{ab}=\Psi_{i_1i_2}\Sigma_{j_1j_2}$ and $(\Omega^{-1})_{aa}=(\Psi^{-1})_{i_1i_1}(\Sigma^{-1})_{j_1j_1}$, we recover inequalities~\eqref{cltbd1}~and~\eqref{cltbd2}.
\end{remark}

\subsection{Example 2: Comparison between the centered matrix \texorpdfstring{$T$}{T} and normal distributions}\label{sec:matrix.T}

The centered matrix $T$ distribution plays a central role in matrix-variate modeling: it arises naturally as an inverse-Wishart scale mixture of a matrix normal in the context of predictive distributions in the matrix-variate linear model, and, after standardization by its row and column scale matrices, it has the left/right spherical symmetries characteristic of matrix-variate laws; see \citet{MR614963}. Let us recall the definition of the centered matrix $T$ distribution. Let $n>0$ denote the real degrees-of-freedom parameter, and $\Psi\in\mathcal{S}_{++}^{\nu}$ and $\Sigma\in\mathcal{S}_{++}^{d}$ the row- and column-scale matrices. Adapting Definition~4.2.1 of \citet{MR1738933} with an extra $\!\sqrt{n}$-scaling, a random matrix $\mathfrak{T}_n\in \R^{\nu\times d}$ is said to have the centered matrix $T$ distribution of degree $n$, written $\mathfrak{T}_n\sim \mathcal{T}_{\nu\times d}(n,0_{\nu\times d},\Psi\otimes\Sigma)$, if its density is given, for all $X\in \R^{\nu\times d}$, by
\begin{equation}\label{eq:matrix.T}
p_{\mathfrak{T}_n}(X)
= \frac{\Gamma_{\nu}((n+\nu+d-1)/2) \big|I_{\nu} + n^{-1} \Psi^{-1} X \Sigma^{-1} X^{\top}\big|^{-(n+\nu+d-1)/2}}{(\pi n)^{\nu d/2} \Gamma_{\nu}((n+\nu-1)/2) |\Psi|^{d/2} |\Sigma|^{\nu/2}},
\end{equation}
where $\Gamma_{\nu}$ denotes the multivariate gamma function. Because of the extra $\!\sqrt{n}$-scaling, this definition has an extra $n^{-(\nu d)/2}$ factor in the normalization constant compared to Definition~4.2.1 of \citet{MR1738933}. However, this version is the direct generalization of the univariate Student distribution; in particular, we have $\mathfrak{T}_n\rightsquigarrow \mathcal{N}_{\nu\times d}(0_{\nu\times d},\Psi \otimes \Sigma)$ as $n\to \infty$.

In Proposition~\ref{prop:T.vs.normal} below, we derive, for $n>2$, an explicit Wasserstein distance bound to quantify the matrix normal approximation of the matrix $T$ distribution. We derive our bound using the comparison-of-Stein-operators approach; an exposition of this approach in the univariate case is given in the survey of \citet{lrs17}, with one of their examples being the derivation of total variation distance bounds for the normal approximation of Student's $t$-distribution. As part of our derivation, we construct a Stein operator for the matrix $T$ law, via its Langevin diffusion.

\begin{proposition}[Wasserstein distance bound between the matrix $T$ and matrix normal]\label{prop:T.vs.normal}
Let $\mathfrak{Z}\sim\mathcal{N}_{\nu\times d}(0_{\nu\times d},\Psi\otimes\Sigma)$ and $\mathfrak{T}_n\sim\mathcal{T}_{\nu\times d}(n,0_{\nu\times d},\Psi\otimes\Sigma)$ for some real $n > 2$. Denote the probability measures of $\mathfrak{T}_n$ and $\mathfrak{Z}$ by $\tau_n$ and $\gamma$, respectively. Then
\begin{equation}\label{tbound}
d_{\mathrm{W}}(\tau_n,\gamma) \leq \frac{10}{\!\sqrt{1-2/n}} \, \frac{(\nu d)^2}{n} \|\Psi\|_2^{1/2} \|\Sigma\|_2^{1/2}.
\end{equation}
For fixed $\nu$, $d$, $\Psi$ and $\Sigma$, the $n^{-1}$ rate of the bound \eqref{tbound} is optimal.
\end{proposition}

\begin{remark}\label{rem:T.n.condition}
The bound (\ref{tbound}) is derived in the case $n>4$ using the comparison-of-Stein-operators approach, where the condition $n>4$ arises from the fact that in carrying out this approach we are required to calculate fourth-order Frobenius moments of the matrix $T$ distribution, which only exist for $n>4$. We then extend the range of validity of the bound to $n>2$ by complementing this bound with the trivial estimate $d_{\mathrm{W}}(\tau_n,\gamma) \leq \EE[\|\mathfrak{T}_n-\mathfrak{Z}\|_F] \leq \EE[\|\mathfrak{T}_n\|_F]+ \EE[\|\mathfrak{Z}\|_F] \leq \smash{(\EE[\|\mathfrak{T}_n\|_F^2])^{1/2} + (\EE[\|\mathfrak{Z}\|_F^2])^{1/2}}$ and then computing these second-order Frobenius moments, for which $\EE[\|\mathfrak{T}_n\|_F^2]$ only exists for $n>2$. The reason why we did not apply the alternative trivial bound $d_{\mathrm{W}}(\tau_n,\gamma)\leq \EE[\|\mathfrak{T}_n\|_F]+ \EE[\|\mathfrak{Z}\|_F]$ (which holds for $n>1$) is that there is no simple formula for $\EE[\|\mathfrak{T}_n\|_F]$, so we instead elected to apply the Cauchy-Schwarz inequality to bound this moment in terms of $\EE[\|\mathfrak{T}_n\|_F^2]$, for which a simple exact formula is available. It should be noted that the Wasserstein distance only exists for distributions with finite absolute first moments, and since the first absolute moment of Student's $t$-distribution does not exist for $n\leq1$, the Wasserstein distance $d_{\mathrm{W}}(\tau_n,\gamma)$ does not exist for $n\leq1$.
\end{remark}

\begin{remark}
As noted in Section~\ref{sec:intro}, the centered matrix $T$ law is genuinely matrix-variate: its density is governed by the matrix quadratic form
$\Psi^{-1/2}X\Sigma^{-1}X^{\top}\Psi^{-1/2}$ through $|I_{\nu}+n^{-1}\Psi^{-1}X\Sigma^{-1}X^{\top}|$, and it enjoys left/right symmetries that are obscured by naive vectorization; see \citet{MR614963,MR881434}, or \citet[p.~140]{MR1738933}. In particular, when $\nu,d\geq 2$, $\vecc(\mathfrak{T}_n^{\top})$ is not a multivariate $t$ distribution on $\R^{\nu d}$ with covariance matrix $\Psi\otimes\Sigma$, so one cannot simply invoke existing multivariate $t$ Stein machinery after vectorization. Proposition~\ref{prop:T.vs.normal} therefore crucially relies on our matrix Stein framework, via a Stein operator tailored to the matrix $T$ law (constructed from its matrix-valued Langevin diffusion) and compared to the matrix Ornstein--Uhlenbeck operator.
\end{remark}

\begin{remark}
While for fixed $\nu$, $d$, $\Psi$ and $\Sigma$, the $n^{-1}$ rate is optimal, the bound \eqref{tbound} appears to be sub-optimal in other regards. We believe that the dimensional dependence is sub-optimal by comparison with the multivariate case. In the $d=1$ case, setting $\Psi=I_\nu$ and $\Sigma=1$, we have the bound $d_{\mathrm{W}}(\tau_n,\gamma)\leq \!\sqrt{2/\pi}\smash{\!\sqrt{C_{n,\nu}}}\leq \!\sqrt{2/\pi}\smash{\!\sqrt{\nu(\nu+5)}}/(n-4)$, valid for $n>4$, where $C_{n,\nu}=\{\nu+2n\nu/(n-2)+n^2\nu(\nu+2)/((n-2)(n-4))\}/(n-1)^2$. The second inequality has the same leading $n^{-1}$ constant as the exact $C_{n,\nu}$-bound for fixed $\nu$, since $C_{n,\nu}\sim \nu(\nu+5)/n^2$ as $n\to\infty$. The bound $d_{\mathrm{W}}(\tau_n,\gamma)\leq \!\sqrt{2/\pi}\smash{\!\sqrt{C_{n,\nu}}}$ can be obtained by applying inequality (5.5) of \citet{mrrs23} together with the second inequality that they display in (5.7) and the Stein discrepancy calculation $S_2(t_n\,|\,\gamma)=\smash{\!\sqrt{C_{n,\nu}}}$ in the notation of \citet[Example~4.21]{mrrs23}. The Stein discrepancy $S_2(\gamma\,|\,t_n)$ was calculated in \citet[Example~4.21]{mrrs23}, but we require the reversed calculation of $S_2(t_n\,|\,\gamma)$ in order to apply their inequality (5.5) to get a Wasserstein distance bound. In comparison, when \eqref{tbound} is specialized to the multivariate $d=1$ case, with $\Psi=I_\nu$ and $\Sigma=1$, we have the bound $d_{\mathrm{W}}(\tau_n,\gamma)\leq 10\nu^2/(n\!\sqrt{1-2/n})$, meaning that, compared with the preceding multivariate bound, we pick up an additional factor of order $\nu$.

We also do not consider the numerical constant 10 to be sharp; this feature of non-sharp numerical constants seen in applications of Stein's method to multivariate distributional approximations is quite common in the literature. For the sake of a simple compact final bound that holds for all $n>2$, we did not attempt to optimize the constant, but we do note that from our proof of Proposition \ref{prop:T.vs.normal} it can be seen that, if one restricts to $n>4$, \eqref{tbound} can be replaced by $d_{\mathrm{W}}(\tau_n,\gamma) \leq \smash{(2\!\sqrt{3}+\!\sqrt{2})(\nu d)^2\|\Psi\|_2^{1/2}\|\Sigma\|_2^{1/2}/(n\!\sqrt{1-4/n})}$, where $2\!\sqrt{3}+\!\sqrt{2}\approx 4.88$. If one instead insists on retaining the denominator $\!\sqrt{1-2/n}$ from \eqref{tbound}, the same proof gives the constant $6+\!\sqrt{6}\approx 8.45$ for $n\geq 5$. To gauge the quality of the constant 10, we note that in the univariate case $\nu=d=1$ (and for simplicity $\Psi=\Sigma=1$) a slight modification of the derivation of the total variation distance bound (79) of \citet{lrs17} that employs a different regularity estimate for the solution of the standard normal Stein equation yields the bound $d_{\mathrm{W}}(\tau_n,\gamma)\leq 2\!\sqrt{2/\pi}/(n-2)\approx 1.596/(n-2)$. From our proof of the optimal $n^{-1}$ rate given in Proposition \ref{prop:T.vs.normal} in which a lower bound is derived for the Wasserstein distance $d_{\mathrm{W}}(\tau_n,\gamma)$, we know that the leading asymptotic constant in an $n^{-1}$ bound cannot be reduced below $(3\!\sqrt{e}-4)/(2\!\sqrt{2 e\pi})\approx0.1144$. A direction for future research is to develop a theory of Stein kernels for matrix distributions; a natural direction would be to attain a Wasserstein distance bound for the matrix normal approximation of the matrix $T$ distribution that improves on the bound \eqref{tbound}.
\end{remark}

\subsection{Example 3: Estimation of the parameters of a centered matrix normal distribution}\label{sec:parameter.estimation}

Estimation of the Kronecker factors in the matrix normal model is classically handled by the flip--flop maximum-likelihood estimator (MLE) of \citet{doi:10.1080/00949659908811970}, which alternates closed-form updates for $\Psi$ and $\Sigma$ under a scale constraint. Subsequent refinements of this approach bundle algorithmic least-squares/MLE improvements for Kronecker-structured covariances together with sparse or regularized formulations (graphical structure and convergence guarantees); see \citet{doi:10.1109/TSP.2007.907834}, \citet{MR2890437}, and \citet{doi:10.1109/TSP.2013.2240157}. The derivation below proceeds by generalizing the centered case ($\mu=0$) in Example~3.6 of \citet{MR4986908}, which uses Stein's method to estimate the variance of a univariate normal, to the matrix-normal setting. Within this generalization, testing the Stein identity with scalar quadratic probes indexed by matrix weights yields moment equations that both recover Dutilleul's flip--flop equations and, more broadly, generate a flexible class of weighted flip--flop Stein estimators.

\begin{remark}[Why the Kronecker structure matters here]\label{rem:app3.matrix.necessary}
If we set $\bb{X}=\vecc(\mathfrak{X}^{\top})$ and write $\Omega=\Psi\otimes\Sigma$, then $\bb{X}\sim\mathcal{N}_{\nu d}(\bb{0}_{\nu d},\Omega)$. For the \emph{unconstrained} multivariate normal model (i.e., without the Kronecker covariance structure), the quadratic Stein identity tested with $f_A(\bb{X}) = (1/2) \bb{X}^{\top} A \bb{X}$ (for $A\in\mathcal{S}^{\nu d}$) yields moment equations of the form $\EE[\bb{X}^{\top}A \bb{X}] = \tr(A\Omega)$, which, after replacing $\EE$ by the empirical average, produces an estimator of the full covariance matrix $\Omega$. However, this does \emph{not} address the separable estimation problem: $\Omega$ must equal $\Psi\otimes\Sigma$ and the pair $(\Psi,\Sigma)$ is identifiable only up to the rescaling $(\Psi,\Sigma)\mapsto(c\Psi,c^{-1}\Sigma)$. In this sense, the Kronecker-aware matrix formulation is not merely a notational convenience: it keeps the method-of-moments Stein equations on $\mathcal{S}^{\nu}\times\mathcal{S}^{d}$ rather than on $\mathcal{S}^{\nu d}$.
\end{remark}

Let $\smash{\Psi\in \mathcal{S}_{++}^{\nu}}$ and $\smash{\Sigma\in \mathcal{S}_{++}^d}$. Given a random sample $\smash{\mathfrak{X}^{(1)},\ldots,\mathfrak{X}^{(n)} \stackrel{\mathrm{iid}}{\sim} \mathcal{N}_{\nu\times d}(0_{\nu\times d}, \Psi \otimes \Sigma)}$ and any function $g: \R^{\nu\times d}\to \R$, define the averaging operator
\[
\overline{g(\mathfrak{X})} = \frac{1}{n} \sum_{k=1}^n g(\mathfrak{X}^{(k)}).
\]
We apply a method of moments to the Stein characterization in Corollary~\ref{cor:Stein.normal.step.1} by replacing the expectation with the empirical average and imposing the resulting empirical Stein equations for selected probe functions $f\in C_{\mathcal{A}^{\mathrm{OU}}}^2(\R^{\nu\times d})$; namely,
\begin{equation}\label{eq:scalar.identity}
\overline{\mathcal{A}^{\mathrm{OU}} f(\mathfrak{X})}
= - \overline{\tr\{\mathfrak{X}^{\top} \nabla f(\mathfrak{X})\}} + \overline{\tr\{\Sigma \nabla^{\top} \Psi \nabla f(\mathfrak{X})\}} = 0.
\end{equation}

Introduce the left--right weighted quadratic probes; for symmetric $W\in \mathcal{S}^{\nu}$ and $U\in \mathcal{S}^{d}$, set
\begin{equation}\label{eq:weighted.probe}
f_{W,U}(X) = \frac{1}{2} \tr(X^{\top}\, \Psi^{-1/2} W \Psi^{-1/2}\, X \, U).
\end{equation}
The geometry of \eqref{eq:weighted.probe} is simple. Write heuristically $\mathfrak{X}=\Psi^{1/2}\mathfrak{Z}\Sigma^{1/2}$ with $\mathfrak{Z}$ having independent standard normal entries. The factor $\Psi^{-1/2}$ whitens the row covariance, so $W$ selects directions in the whitened row space, while $U$ tests directions in the column space. For example, if $W=vv^{\top}$, then $\mathfrak{X}^{\top}\Psi^{-1/2}W\Psi^{-1/2}\mathfrak{X}$ is the column scatter matrix formed from the single whitened row contrast $v^{\top}\Psi^{-1/2}\mathfrak{X}$. The sample average of this scatter estimates $\Sigma$, up to the normalization $\tr(W)$. Exchanging rows and columns gives the analogous estimate of $\Psi$.

A direct calculation shows that $\nabla f_{W,U}(X) = \Psi^{-1/2} W \Psi^{-1/2} \, X \, U$ and
\[
\frac{\partial^2 f_{W,U}(X)}{\partial X_{ij}\,\partial X_{i'j'}} = (\Psi^{-1/2} W \Psi^{-1/2})_{ii'}\, U_{j'j}.
\]
Hence
\[
\tr\{\Sigma \nabla^{\top}\Psi \nabla f_{W,U}(X)\}
= \sum_{i,i',j,j'} \Sigma_{jj'} \Psi_{ii'} (\Psi^{-1/2} W \Psi^{-1/2})_{ii'} U_{j'j}
= \tr(W)\, \tr(\Sigma U),
\]
and
\[
\tr\{X^{\top}\nabla f_{W,U}(X)\}
= \tr\big(X^{\top}\, \Psi^{-1/2} W \Psi^{-1/2}\, X \, U\big).
\]
Substituting $f_{W,U}$ from \eqref{eq:weighted.probe} into \eqref{eq:scalar.identity}, we obtain, for every $U\in \mathcal{S}^{d}$,
\[
\tr\left\{\left(\frac{1}{n}\sum_{k=1}^n (\mathfrak{X}^{(k)})^{\top}\, \Psi^{-1/2} W \Psi^{-1/2}\, \mathfrak{X}^{(k)}\right) U\right\}
= \tr(W)\, \tr(\Sigma U).
\]
Since $U$ ranges over $\mathcal{S}^{d}$, trace duality yields the operator identity
\begin{equation}\label{eq:operator.identity.Sigma}
\frac{1}{n}\sum_{k=1}^n (\mathfrak{X}^{(k)})^{\top}\, \Psi^{-1/2} W \Psi^{-1/2}\, \mathfrak{X}^{(k)} = \tr(W)\, \Sigma.
\end{equation}
The row-side identity is analogous. Using the probes
\[
\widetilde{f}_{W,U}(X) = \frac{1}{2} \tr\big(X \, \Sigma^{-1/2} U \Sigma^{-1/2} \, X^{\top} W\big),
\]
the same steps lead to
\begin{equation}\label{eq:operator.identity.Psi}
\frac{1}{n}\sum_{k=1}^n \mathfrak{X}^{(k)} \, \Sigma^{-1/2} U \Sigma^{-1/2}\, (\mathfrak{X}^{(k)})^{\top} = \tr(U)\, \Psi.
\end{equation}
The unknown factor used for whitening one side is the factor that is held fixed while estimating the other side. This is why the resulting algorithm has the flip--flop form: a current value of $\Psi$ turns the observed matrices into column scatter matrices and hence updates $\Sigma$, while the new value of $\Sigma$ turns them into row scatter matrices and hence updates $\Psi$.

If, in addition, $W\in \mathcal{S}_+^{\nu}$ and $U\in \mathcal{S}_+^{d}$ with $\tr(W)\neq 0$ and $\tr(U)\neq 0$, then dividing \eqref{eq:operator.identity.Sigma} and \eqref{eq:operator.identity.Psi} by $\tr(W)$ and $\tr(U)$ respectively yields the following weighted flip--flop updates: starting from $\Psi^{(0)}\in \mathcal{S}_{++}^{\nu}$, for $t=0,1,2,\ldots$,
\begin{align}
\Sigma^{(t+1)}
&= \frac{1}{n\, \tr(W)} \sum_{k=1}^n (\mathfrak{X}^{(k)})^{\top}\, \big(\Psi^{(t)}\big)^{-1/2} W \big(\Psi^{(t)}\big)^{-1/2}\, \mathfrak{X}^{(k)}, \label{eq:w.Sigma.update} \\
\Psi^{(t+1)}
&= \frac{1}{n\, \tr(U)} \sum_{k=1}^n \mathfrak{X}^{(k)} \, \big(\Sigma^{(t+1)}\big)^{-1/2} U \big(\Sigma^{(t+1)}\big)^{-1/2}\, (\mathfrak{X}^{(k)})^{\top}. \label{eq:w.Psi.update}
\end{align}
This is optionally followed at each step by a rescaling to enforce an identifiability constraint (e.g., $\tr(\Sigma^{(t+1)})=d$ or $|\Sigma^{(t+1)}|=1$, with inverse scaling applied to $\Psi^{(t+1)}$).

\begin{remark}
By setting $W = I_{\nu}$ and $U = I_d$ in \eqref{eq:w.Sigma.update}--\eqref{eq:w.Psi.update}, the weighted scheme reduces to the flip--flop maximum-likelihood estimator of \citet{doi:10.1080/00949659908811970}.
\end{remark}

\begin{remark}[Positive definite iterates]\label{rem:app3.positive.definite.iterates}
The unregularized updates in \eqref{eq:w.Sigma.update} and \eqref{eq:w.Psi.update} are understood on the event that all inverse square roots are well-defined. Thus, starting from $\Psi^{(0)}\in\mathcal{S}_{++}^{\nu}$, one assumes inductively that $\Sigma^{(t+1)}\in\mathcal{S}_{++}^{d}$ before forming $\smash{\big(\Sigma^{(t+1)}\big)^{-1/2}}$, and that $\Psi^{(t+1)}\in\mathcal{S}_{++}^{\nu}$ before the next iteration. Since $W\in\mathcal{S}_{+}^{\nu}$ and $U\in\mathcal{S}_{+}^{d}$, the right-hand sides are always nonnegative definite; they are positive definite exactly under the corresponding finite-sample full-rank conditions. If such a condition fails, one may use a ridge version, replacing the two right-hand sides by the same expressions plus $\varepsilon I_d$ and $\varepsilon I_{\nu}$, respectively, with $\varepsilon>0$, followed if desired by the same identifiability rescaling. This keeps the iterates in the positive definite cones and is a finite-sample numerical regularization, not a modification of the population Stein moment equations above.
\end{remark}

The weighted formulation \eqref{eq:w.Sigma.update}--\eqref{eq:w.Psi.update} admits several immediate uses. First, choosing diagonal $W$ and $U$ allows row/column importance weighting and, after passing to the observed marginal, simple missingness masks, as the following example illustrates.

\begin{example}[Row/column missingness, $\nu=d=2$]\label{example:flip.flop.missing}
A common practical situation is that an entire row or column of each observation is systematically unavailable. For instance, suppose that for every replicate $\mathfrak{X}\in\R^{2\times 2}$, we only observe the second row (e.g., sensor~1 failed, or a subject-level feature cannot be recorded). Writing
\[
P=\mathrm{diag}(0,1), \qquad Q=I_2,
\]
the row mask $P$ reflects which rows are observed (here, only row 2) and the column mask~$Q$ reflects which columns are observed (here, all columns). Write the actually observed data as
\[
\mathfrak{Y}^{(k)} = P \mathfrak{X}^{(k)} Q, \qquad k\in \{1,\ldots,n\}.
\]
Since only the observed marginal is available, the computable moment equations use the covariance blocks of the observed rows and columns. With $(P\Psi P)^+$ and $(Q\Sigma Q)^+$ denoting Moore-Penrose inverses, they are
\[
\frac{1}{n}\sum_{k=1}^n \big(\mathfrak{Y}^{(k)}\big)^{\top} (P\Psi P)^+ \mathfrak{Y}^{(k)}
= \tr(P)\, Q\Sigma Q, \qquad
\frac{1}{n}\sum_{k=1}^n \mathfrak{Y}^{(k)} (Q\Sigma Q)^+ \big(\mathfrak{Y}^{(k)}\big)^{\top}
= \tr(Q)\, P\Psi P.
\]
Thus, in the present row-masking case, where $\tr(P)=1$ and $Q=I_2$, the natural targets are the full column scale $\Sigma$ and the identifiable row block $P\Psi P$ of the row scale, under the chosen scale convention. In particular, only the $(2,2)$ principal entry of $\Psi$ is identifiable; the off-diagonal and the first diagonal entry cannot be recovered from the data and are not estimated. Operationally, starting from any $\Psi^{(0)}\in\mathcal{S}_{++}^{2}$, the masked flip--flop updates are
\[
\Sigma^{(t+1)}
= \frac{1}{n\,\tr(P)} \sum_{k=1}^n \big(\mathfrak{Y}^{(k)}\big)^{\top} \big(P\Psi^{(t)}P\big)^+ \mathfrak{Y}^{(k)}, \qquad
P\Psi^{(t+1)}P
= \frac{1}{n\,\tr(Q)} \sum_{k=1}^n \mathfrak{Y}^{(k)} \big(Q\Sigma^{(t+1)}Q\big)^+ \big(\mathfrak{Y}^{(k)}\big)^{\top},
\]
followed by an identifiability rescaling ($\tr(Q\Sigma^{(t+1)}Q)=\tr(Q)$ and inverse scaling on $P\Psi^{(t+1)}P$). The same template handles a simple column missingness pattern by taking, for example, $Q=\mathrm{diag}(0,1)$ and $P=I_2$. In that case only the $(2,2)$ principal entry of $\Sigma$ is identifiable and we report $\smash{Q\widehat{\Sigma}Q}$, while $\Psi$ is estimated under the same scale convention.

This example clarifies the role of the weights in the missingness setting: the moment equations are written for the observed marginal, so the row mask $P$ suppresses the systematically unobserved rows and, symmetrically, the column mask $Q$ suppresses systematically unobserved columns. At the population level, the row-mask equations with $Q=I_d$ target $\Sigma$ and $P\Psi P$ under the chosen scale convention; the unobserved entries of $\Psi$ are not identified. Similarly, the column-mask equations with $P=I_{\nu}$ target $\Psi$ and $Q\Sigma Q$ under the chosen scale convention; the unobserved entries of $\Sigma$ are not identified. An \textsf{R} implementation of these masked flip--flop updates is provided in \citet{GauntOuimetRichards2026github}.
\end{example}

Second, by choosing several weight matrices $\{W_r\}_{r=1}^{R}$ and $\{U_m\}_{m=1}^{M}$ and feeding them into the flip--flop identities \eqref{eq:w.Sigma.update}--\eqref{eq:w.Psi.update}, we obtain a family of Stein-type moment equations that, with the relevant whitening factor held fixed, produce multiple unbiased estimators of $\Sigma$ and $\Psi$. These can be averaged and, when we restrict $\Sigma$ and $\Psi$ to structured linear subspaces, projected by least squares onto those subspaces, as detailed next.

\begin{proposition}[Oracle projection onto a structured subspace]\label{prop:projected.Stein}
Suppose that the true covariance factors $\Sigma\in \mathcal{S}_{++}^{d}$ and $\Psi\in \mathcal{S}_{++}^{\nu}$ admit linear representations
\[
\Sigma = \Sigma(\bb{\beta}^{\star}) = \sum_{j=1}^p \beta^{\star}_j B_j, \qquad
\Psi = \Psi(\bb{\alpha}^{\star}) = \sum_{\ell=1}^q \alpha^{\star}_{\ell} A_{\ell},
\]
for $\bb{\beta}^{\star}\in\R^p$, $\bb{\alpha}^{\star}\in\R^q$, and fixed symmetric templates $\smash{\{B_j\}_{j=1}^p} \subseteq \mathcal{S}^d$ and $\smash{\{A_{\ell}\}_{\ell=1}^q} \subseteq \mathcal{S}^{\nu}$. Let the weights $\{W_r\}_{r=1}^{R} \subseteq \mathcal{S}^{\nu}$ and $\{U_m\}_{m=1}^{M} \subseteq \mathcal{S}^{d}$ be given with $\tr(W_r)\neq 0$ and $\tr(U_m)\neq 0$. With the whitening factors held fixed at their true values, define
\[
M_r(\Psi) = \frac{1}{n\, \tr(W_r)} \sum_{k=1}^n (\mathfrak{X}^{(k)})^{\top}\, \Psi^{-1/2} W_r \Psi^{-1/2}\, \mathfrak{X}^{(k)}, \qquad
N_m(\Sigma) = \frac{1}{n\, \tr(U_m)} \sum_{k=1}^n \mathfrak{X}^{(k)} \, \Sigma^{-1/2} U_m \Sigma^{-1/2}\, (\mathfrak{X}^{(k)})^{\top}.
\]
For each $m\in \{1,\ldots,M\}$ and $r\in \{1,\ldots,R\}$, form the structured Stein moment equations
\begin{equation}\label{eq:structured.Sigma.moments}
\tr\{\Sigma(\bb{\beta})\, U_m\} = \frac{1}{R} \sum_{r=1}^{R} \tr\{M_r(\Psi)\, U_m\} \equiv y^{(\Sigma)}_m, \qquad
\tr\{\Psi(\bb{\alpha})\, W_r\} = \frac{1}{M} \sum_{m=1}^{M} \tr\{N_m(\Sigma)\, W_r\} \equiv y^{(\Psi)}_r.
\end{equation}
Here $\smash{\bb{y}^{(\Sigma)}=(y^{(\Sigma)}_1,\ldots,y^{(\Sigma)}_M)^{\top}}$ and $\smash{\bb{y}^{(\Psi)}=(y^{(\Psi)}_1,\ldots,y^{(\Psi)}_R)^{\top}}$.
Let $\smash{C^{(\Sigma)} \in \R^{M\times p}}$ and $\smash{C^{(\Psi)} \in \R^{R\times q}}$ be the design matrices with entries
\[
C^{(\Sigma)}_{mj}=\tr(B_j U_m), \qquad C^{(\Psi)}_{r\ell}=\tr(A_{\ell} W_r).
\]
If $\smash{C^{(\Sigma)}}$ and $\smash{C^{(\Psi)}}$ have full column rank, then the least-squares solutions
\begin{equation}\label{eq:structured.Stein.least.squares}
\widehat{\bb{\beta}} = \arg\min_{\bb{\beta}\in\R^p} \big\|C^{(\Sigma)} \bb{\beta} - \bb{y}^{(\Sigma)}\big\|_2^2, \qquad
\widehat{\bb{\alpha}} = \arg\min_{\bb{\alpha}\in\R^q} \big\|C^{(\Psi)} \bb{\alpha} - \bb{y}^{(\Psi)}\big\|_2^2,
\end{equation}
are strongly consistent for the true coefficients $\bb{\beta}^{\star}$ and $\bb{\alpha}^{\star}$, as $n\to\infty$. In particular, the oracle structured estimates $\smash{\widehat{\Sigma}=\Sigma(\widehat{\bb{\beta}})}$ and $\smash{\widehat{\Psi}=\Psi(\widehat{\bb{\alpha}})}$ converge almost surely to $\Sigma$ and $\Psi$, and they solve the projected Stein moment equations \eqref{eq:structured.Sigma.moments} in the least-squares sense.
\end{proposition}

\begin{remark}[Intuition behind Proposition~\ref{prop:projected.Stein}]
The first display in Proposition~\ref{prop:projected.Stein} contains the structural assumption. Instead of estimating the $d(d+1)/2$ free entries of $\Sigma$ and the $\nu(\nu+1)/2$ free entries of $\Psi$ separately, we assume that these matrices lie in two prescribed linear spaces, namely $\mathrm{span}\{B_1,\ldots,B_p\}\subseteq\mathcal{S}^d$ and $\mathrm{span}\{A_1,\ldots,A_q\}\subseteq\mathcal{S}^{\nu}$. The unknowns are therefore the coefficient vectors $\bb{\beta}^{\star}$ and $\bb{\alpha}^{\star}$, not all matrix entries. In this sense, the templates $B_j$ and $A_{\ell}$ play the same role as regressors in an ordinary linear model: they are fixed directions, and the coefficients tell us how much of each direction is present in the true covariance factors.

The matrices $U_m$ and $W_r$ are not additional covariance parameters. They are test directions, or instruments, used to read off scalar coordinates of matrices through the trace inner product. Thus $\tr\{\Sigma U_m\}$ is the coordinate of $\Sigma$ seen in the column-space direction $U_m$, while $\tr\{\Psi W_r\}$ is the coordinate of $\Psi$ seen in the row-space direction $W_r$. The same trace coordinates turn the structural representations into linear equations:
\[
\tr\{\Sigma(\bb{\beta})\, U_m\} = \sum_{j=1}^p \beta_j\, \tr(B_j U_m), \qquad
\tr\{\Psi(\bb{\alpha})\, W_r\} = \sum_{\ell=1}^q \alpha_{\ell}\, \tr(A_{\ell} W_r).
\]
This display explains why estimating structured matrices reduces to estimating the two coefficient vectors $\bb{\beta}^{\star}$ and $\bb{\alpha}^{\star}$. Once the probes $U_m$ and $W_r$ have been chosen, the left-hand sides are just linear predictions in the unknown coefficients.

The specific forms of $M_r(\Psi)$ and $N_m(\Sigma)$ come from \eqref{eq:operator.identity.Sigma} and \eqref{eq:operator.identity.Psi}, equivalently from the weighted flip--flop updates \eqref{eq:w.Sigma.update}--\eqref{eq:w.Psi.update} after replacing $W$ by $W_r$ and $U$ by $U_m$. If the row covariance factor $\Psi$ is known and we whiten the rows by setting $\mathfrak{Z}_{\mathrm{row}}=\Psi^{-1/2}\mathfrak{X}$, then $\mathfrak{Z}_{\mathrm{row}}$ has row scale $I_{\nu}$ and column scale $\Sigma$, and
\[
\EE[\mathfrak{Z}_{\mathrm{row}}^{\top} W_r \mathfrak{Z}_{\mathrm{row}}] = \tr(W_r)\,\Sigma.
\]
Replacing the expectation by the sample average and dividing by $\tr(W_r)$ gives $M_r(\Psi)$. Thus $W_r$ appears inside $M_r(\Psi)$ because $W_r$ selects how the whitened rows are contracted, and after this row-side contraction the remaining object is a $d\times d$ matrix estimating the column factor $\Sigma$. Similarly, if the column covariance factor $\Sigma$ is known and we whiten the columns by setting $\mathfrak{Z}_{\mathrm{col}}=\mathfrak{X}\Sigma^{-1/2}$, then $\mathfrak{Z}_{\mathrm{col}}$ has row scale $\Psi$ and column scale $I_d$, and
\[
\EE[\mathfrak{Z}_{\mathrm{col}} U_m \mathfrak{Z}_{\mathrm{col}}^{\top}] = \tr(U_m)\,\Psi.
\]
Replacing the expectation by the sample average and dividing by $\tr(U_m)$ gives $N_m(\Sigma)$. Thus $U_m$ appears inside $N_m(\Sigma)$ because $U_m$ selects how the whitened columns are contracted, and after this column-side contraction the remaining object is a $\nu\times\nu$ matrix estimating the row factor $\Psi$. This is the main source of the apparent crossing of indices: row probes produce estimates of the column factor, and column probes produce estimates of the row factor.

The scalars $\smash{y_m^{(\Sigma)}}$ and $\smash{y_r^{(\Psi)}}$ are the observed responses in these trace coordinates. In the $\Sigma$-equation in \eqref{eq:structured.Sigma.moments}, the left-hand side $\tr\{\Sigma(\bb{\beta})\, U_m\}$ is the model prediction for the $m$th coordinate of the structured column factor, while the right-hand side is the average of several data-based estimates of the same coordinate:
\[
y_m^{(\Sigma)} = \frac{1}{R}\sum_{r=1}^R \tr\{M_r(\Psi)\, U_m\}.
\]
For fixed $m$, each term $\tr\{M_r(\Psi)\, U_m\}$ estimates $\tr\{\Sigma U_m\}$, but it does so through a different row instrument $W_r$. Averaging over $r$ therefore pools $R$ empirical Stein moments with the same population target. In the $\Psi$-equation in \eqref{eq:structured.Sigma.moments}, the roles are interchanged:
\[
y_r^{(\Psi)} = \frac{1}{M}\sum_{m=1}^M \tr\{N_m(\Sigma)\, W_r\}.
\]
For fixed $r$, each term $\tr\{N_m(\Sigma)\, W_r\}$ estimates $\tr\{\Psi W_r\}$, but it does so through a different column instrument $U_m$. Averaging over $m$ pools $M$ empirical Stein moments with the same population target. The averages over $k=1,\ldots,n$ inside $M_r(\Psi)$ and $N_m(\Sigma)$ are the ordinary sample averages over the observed matrices $\mathfrak{X}^{(k)}$.

The design matrices $\smash{C^{(\Sigma)}}$ and $\smash{C^{(\Psi)}}$ collect the deterministic parts of these two regressions. The entry in row $m$ and column $j$ of $\smash{C^{(\Sigma)}}$ is the trace inner product $\tr(B_jU_m)$; it tells how much the $j$th template for $\Sigma$ contributes to the $m$th observed coordinate. The entry in row $r$ and column $\ell$ of $\smash{C^{(\Psi)}}$ is $\tr(A_{\ell}W_r)$, with the analogous interpretation for $\Psi$. Hence $\smash{C^{(\Sigma)}\bb{\beta}}$ is the vector of fitted trace coordinates produced by the structured candidate $\Sigma(\bb{\beta})$, and $\smash{C^{(\Psi)}\bb{\alpha}}$ is the vector of fitted trace coordinates produced by $\Psi(\bb{\alpha})$. The design matrices are coordinate maps from coefficient space into moment-coordinate space. The projection occurs when least squares replaces the observed coordinate vectors by their closest fitted coordinate vectors in the column spaces of these maps.

This gives a regression formulation at the level of moments. Set
\[
\bb{\theta}^{(\Sigma)}=(\tr\{\Sigma U_1\},\ldots,\tr\{\Sigma U_M\})^{\top}, \qquad \bb{\theta}^{(\Psi)}=(\tr\{\Psi W_1\},\ldots,\tr\{\Psi W_R\})^{\top}.
\]
At the population level one has
\[
\bb{\theta}^{(\Sigma)} = C^{(\Sigma)}\bb{\beta}^{\star}, \qquad
\bb{\theta}^{(\Psi)} = C^{(\Psi)}\bb{\alpha}^{\star},
\]
and at finite sample size one observes
\[
\bb{y}^{(\Sigma)} = C^{(\Sigma)}\bb{\beta}^{\star} + \bb{\varepsilon}^{(\Sigma)}_n, \qquad
\bb{y}^{(\Psi)} = C^{(\Psi)}\bb{\alpha}^{\star} + \bb{\varepsilon}^{(\Psi)}_n,
\]
where the error terms are centered empirical quadratic fluctuations. Thus $\smash{\bb{y}^{(\Sigma)}}$ and $\smash{\bb{y}^{(\Psi)}}$ are the response vectors, $\smash{C^{(\Sigma)}\bb{\beta}}$ and $\smash{C^{(\Psi)}\bb{\alpha}}$ are the fitted responses allowed by the structured covariance model, and $\smash{\widehat{\bb{\beta}}}$ and $\smash{\widehat{\bb{\alpha}}}$ are the least-squares coefficient estimates. If the two design matrices have full column rank, then the fitted moment vectors are
\[
C^{(\Sigma)}\widehat{\bb{\beta}} = P_{C^{(\Sigma)}}\bb{y}^{(\Sigma)}, \qquad
C^{(\Psi)}\widehat{\bb{\alpha}} = P_{C^{(\Psi)}}\bb{y}^{(\Psi)},
\]
where $\smash{P_{C^{(\Sigma)}} = C^{(\Sigma)}(\big(C^{(\Sigma)}\big)^{\top}C^{(\Sigma)})^{-1}\big(C^{(\Sigma)}\big)^{\top}}$ and $\smash{P_{C^{(\Psi)}} = C^{(\Psi)}(\big(C^{(\Psi)}\big)^{\top}C^{(\Psi)})^{-1}\big(C^{(\Psi)}\big)^{\top}}$. Thus the final calculation in \eqref{eq:structured.Stein.least.squares} is the usual normal-equations calculation from linear regression. If $M=p$ and $\smash{C^{(\Sigma)}}$ is nonsingular, the $\Sigma$ moment equations are solved exactly; if $M>p$, they are overidentified and the least-squares solution chooses the structured covariance whose trace coordinates are closest to the observed vector $\smash{\bb{y}^{(\Sigma)}}$. The same interpretation applies to $\Psi$. If $\{U_m\}_{m=1}^{M}$ and $\{W_r\}_{r=1}^{R}$ are complete orthonormal bases of $\mathcal{S}^d$ and $\mathcal{S}^{\nu}$ for the trace inner product, this coordinate projection is the usual Frobenius projection:
\[
\Sigma(\widehat{\bb{\beta}})=\arg\min_{S\in\mathrm{span}\{B_1,\ldots,B_p\}}\left\|S-\sum_{m=1}^M y_m^{(\Sigma)}U_m\right\|_F^2, \qquad
\Psi(\widehat{\bb{\alpha}})=\arg\min_{T\in\mathrm{span}\{A_1,\ldots,A_q\}}\left\|T-\sum_{r=1}^R y_r^{(\Psi)}W_r\right\|_F^2.
\]
For general probes, the same normal equations project the response vectors $\smash{\bb{y}^{(\Sigma)}}$ and $\smash{\bb{y}^{(\Psi)}}$ onto the column spaces of $\smash{C^{(\Sigma)}}$ and $\smash{C^{(\Psi)}}$ in Euclidean norm on the selected moment coordinates.

Finally, the advantage of staying with the row and column covariance factors $\Psi$ and $\Sigma$ is that the projection remains linear. If one instead vectorizes the full covariance $\Omega=\Psi\otimes\Sigma$ and imposes the same structures, then
\[
\Omega(\bb{\alpha},\bb{\beta}) = \Big(\sum_{\ell=1}^q \alpha_{\ell}A_{\ell}\Big)\otimes\Big(\sum_{j=1}^p \beta_jB_j\Big),
\]
which is bilinear in $(\bb{\alpha},\bb{\beta})$. The weighted Stein formulation avoids this by alternating between two linear factor-space regressions, followed in practice by the usual identifiability rescaling of the two covariance factors.
\end{remark}

\begin{remark}[Projected flip--flop implementation]\label{rem:projected.Stein.iteration}
Proposition~\ref{prop:projected.Stein} records the projected Stein equations with one covariance factor held fixed while the other one is updated. In practice, the vectors $\smash{\bb{y}^{(\Sigma)}}$ and $\smash{\bb{y}^{(\Psi)}}$ are therefore recomputed along a flip--flop iteration, exactly as in \eqref{eq:w.Sigma.update}--\eqref{eq:w.Psi.update}, except that each raw matrix update is followed by a least-squares projection onto the linear spaces $\mathrm{span}\{B_1,\ldots,B_p\}$ and $\mathrm{span}\{A_1,\ldots,A_q\}$. Starting from $\Psi^{(0)}\in \mathcal{S}_{++}^{\nu}$, for $t=0,1,2,\ldots$, define
\[
\begin{aligned}
M_r(\Psi^{(t)})
&= \frac{1}{n\, \tr(W_r)} \sum_{k=1}^n (\mathfrak{X}^{(k)})^{\top}\, \big(\Psi^{(t)}\big)^{-1/2} W_r \big(\Psi^{(t)}\big)^{-1/2}\, \mathfrak{X}^{(k)}, \qquad r=1,\ldots,R, \\
y^{(\Sigma,t+1)}_m
&= \frac{1}{R} \sum_{r=1}^{R}\tr\{M_r(\Psi^{(t)})\, U_m\}, \qquad m=1,\ldots,M,
\end{aligned}
\]
and set
\[
\bb{\beta}^{(t+1)} = \arg\min_{\bb{\beta}\in\R^p} \big\|C^{(\Sigma)} \bb{\beta} - \bb{y}^{(\Sigma,t+1)}\big\|_2^2, \qquad \Sigma^{(t+1)}=\Sigma(\bb{\beta}^{(t+1)}).
\]
With this new value of $\Sigma$, define
\[
\begin{aligned}
N_m(\Sigma^{(t+1)})
&= \frac{1}{n\, \tr(U_m)} \sum_{k=1}^n \mathfrak{X}^{(k)}\, \big(\Sigma^{(t+1)}\big)^{-1/2} U_m \big(\Sigma^{(t+1)}\big)^{-1/2}\, (\mathfrak{X}^{(k)})^{\top}, \qquad m=1,\ldots,M, \\
y^{(\Psi,t+1)}_r
&= \frac{1}{M} \sum_{m=1}^{M}\tr\{N_m(\Sigma^{(t+1)})\, W_r\}, \qquad r=1,\ldots,R,
\end{aligned}
\]
and set
\[
\bb{\alpha}^{(t+1)} = \arg\min_{\bb{\alpha}\in\R^q} \big\|C^{(\Psi)} \bb{\alpha} - \bb{y}^{(\Psi,t+1)}\big\|_2^2, \qquad \Psi^{(t+1)}=\Psi(\bb{\alpha}^{(t+1)}).
\]
Since $\smash{C^{(\Sigma)}}$ and $\smash{C^{(\Psi)}}$ do not depend on $t$, the least-squares steps can equivalently be written as
\[
\bb{\beta}^{(t+1)} = \big(\big(C^{(\Sigma)}\big)^{\top}C^{(\Sigma)}\big)^{-1}\big(C^{(\Sigma)}\big)^{\top}\bb{y}^{(\Sigma,t+1)}, \qquad \bb{\alpha}^{(t+1)} = \big(\big(C^{(\Psi)}\big)^{\top}C^{(\Psi)}\big)^{-1}\big(C^{(\Psi)}\big)^{\top}\bb{y}^{(\Psi,t+1)},
\]
whenever the two design matrices have full column rank. The iteration may be followed at each step by a rescaling to enforce an identifiability constraint, for instance $\tr(\Sigma^{(t+1)})=d$ or $|\Sigma^{(t+1)}|=1$, with inverse scaling applied to $\Psi^{(t+1)}$. Since the structural spaces are linear, this rescaling simply rescales the corresponding coefficient vectors. Thus the apparent circularity in \eqref{eq:structured.Sigma.moments} is resolved in the same way as for the ordinary weighted flip--flop equations: $\Psi^{(t)}$ is used to form the projected update of $\Sigma^{(t+1)}$, and $\Sigma^{(t+1)}$ is then used to form the projected update of $\Psi^{(t+1)}$.

In high-dimensional or ill-conditioned samples, the projected updates can be stabilized before the next inverse square root is taken. For example, one may replace a generic pair of updates $(\widehat{\Sigma},\widehat{\Psi})$ by
\[
\widehat{\Sigma}_{\lambda} = (1 - \lambda) \widehat{\Sigma} + \lambda \tau I_d, \qquad \widehat{\Psi}_{\lambda} = (1 - \lambda) \widehat{\Psi} + \lambda \tau I_{\nu},
\]
with $0 < \lambda < 1$ and $\tau > 0$ chosen so that $\widehat{\Sigma}_{\lambda}\in\mathcal{S}_{++}^{d}$ and $\widehat{\Psi}_{\lambda}\in\mathcal{S}_{++}^{\nu}$; see, e.g., \citet{MR2026339}. This is automatic if the raw projected updates are nonnegative definite; otherwise $\lambda$ and $\tau$ must be chosen large enough to dominate any negative eigenvalues. This ridge step is a numerical stabilization of the whitening matrices, and it preserves the chosen structural space whenever $I_d\in\mathrm{span}\{B_1,\ldots,B_p\}$ and $I_{\nu}\in\mathrm{span}\{A_1,\ldots,A_q\}$. Alternatively, one may use penalized formulations on the precision factors $\Sigma^{-1}$ and $\Psi^{-1}$, for instance with $\ell_1$ penalties, to obtain positive definite row and column precision estimates; see \citet{MR2890437,doi:10.1109/TSP.2013.2240157}.
\end{remark}

\section{Proofs}\label{sec:proofs}

\begin{proof}[\bf Proof of Proposition~\ref{prop:OU.process.generator}]\pdfbookmark[2]{Proof of Proposition \ref{prop:OU.process.generator}}{proof:OU.process.generator}
For $f\in C^2(\R^{\nu\times d})$, the extended generator of $(\mathfrak{X}_t)_{t\geq 0}$ arises from It\^o's formula as
\[
\begin{aligned}
\mathcal{A}^{\mathrm{OU}}f(X)
&= -\tr\{X^{\top} \nabla f(X)\} + \frac{1}{2} \sum_{i,k=1}^{\nu} \sum_{j,\ell=1}^d \frac{(\rd \mathfrak{X}_t^{(\mathrm{diff})})_{ij} \bullet (\rd \mathfrak{X}_t^{(\mathrm{diff})})_{k\ell}}{\rd t} \nabla_{ij} \nabla_{k\ell} f(X),
\end{aligned}
\]
where $\bullet$ denotes the It\^o product, and
\[
\rd \mathfrak{X}_t^{(\mathrm{diff})} = \!\sqrt{2} \, \Psi^{1/2} \rd \mathfrak{B}_t \Sigma^{1/2}.
\]
Using the fact that $(\rd \mathfrak{B}_t)_{pq} \bullet (\rd \mathfrak{B}_t)_{rs} = \delta_{pr} \delta_{qs} \rd t$, we have
\[
\begin{aligned}
(\rd \mathfrak{X}_t^{(\mathrm{diff})})_{ij} \bullet (\rd \mathfrak{X}_t^{(\mathrm{diff})})_{k\ell}
&= 2 \, \Bigg\{\sum_{p=1}^{\nu} \sum_{q=1}^d (\Psi^{1/2})_{ip} (\rd \mathfrak{B}_t)_{pq} (\Sigma^{1/2})_{qj}\Bigg\} \bullet \Bigg\{\sum_{r=1}^{\nu} \sum_{s=1}^d (\Psi^{1/2})_{kr} (\rd \mathfrak{B}_t)_{rs} (\Sigma^{1/2})_{s\ell}\Bigg\} \\
&= 2 \sum_{p,r=1}^{\nu} \sum_{q,s=1}^d (\Psi^{1/2})_{ip} (\Sigma^{1/2})_{qj} (\Psi^{1/2})_{kr} (\Sigma^{1/2})_{s\ell} \, \{(\rd \mathfrak{B}_t)_{pq} \bullet (\rd \mathfrak{B}_t)_{rs}\} \\
&= 2 \sum_{p=1}^{\nu} \sum_{q=1}^d (\Psi^{1/2})_{ip} (\Psi^{1/2})_{pk} (\Sigma^{1/2})_{jq} (\Sigma^{1/2})_{q\ell} \, \rd t \\
&= 2 \, \Psi_{ik} \Sigma_{j\ell} \, \rd t.
\end{aligned}
\]
Hence, the diffusion part of $\mathcal{A}^{\mathrm{OU}}f(X)$ becomes
\[
\sum_{i,k=1}^{\nu} \sum_{j,\ell=1}^d \Psi_{ik} \Sigma_{j\ell} \nabla_{ij} \nabla_{k\ell} f(X)
= \sum_{j,\ell=1}^d \Sigma_{j\ell} (\nabla^{\top} \Psi \nabla f(X))_{\ell j}
= \tr\{\Sigma \nabla^{\top} \Psi \nabla f(X)\}.
\]
Consequently, the extended generator of $(\mathfrak{X}_t)_{t\geq 0}$ is
\[
\mathcal{A}^{\mathrm{OU}}f(X) = -\tr\{X^{\top} \nabla f(X)\} + \tr\{\Sigma \nabla^{\top} \Psi \nabla f(X)\},
\]
as claimed. This concludes the proof.
\end{proof}

\begin{proof}[\bf Proof of Proposition~\ref{prop:OU.process.distribution}]\pdfbookmark[2]{Proof of Proposition \ref{prop:OU.process.distribution}}{proof:OU.process.distribution}
The arguments here closely follow those in the proofs of Theorems 3.47--3.50 of \citet{arXiv:1201.3256}.

The proof is divided into three steps: Step~1 derives an explicit expression for the unique strong solution of the matrix Ornstein--Uhlenbeck SDE in \eqref{eq:OU.process}; Step~2 obtains the distribution of $\mathfrak{X}_t \mid \{\mathfrak{X}_0 = X_0\}$; and Step~3 takes the limit as $t \to \infty$ to identify the stationary limiting distribution.

{\bf Step~1:} We show below that the unique strong solution of the matrix Ornstein--Uhlenbeck SDE in \eqref{eq:OU.process} is
\[
\mathfrak{X}_t = e^{-t} X_0 + \!\sqrt{2} e^{-t} \int_0^t e^s \Psi^{1/2} \rd \mathfrak{B}_s \Sigma^{1/2}, \qquad \mathfrak{X}_0 = X_0.
\]
Indeed, using the product rule for matrix-variate It\^o processes \citep[see, e.g.,][Lemma~5.11]{MR2353270}, and
\[
\rd \left(\!\sqrt{2} e^{-t}\right) \bullet \rd \left(\int_0^t e^s \Psi^{1/2} \rd \mathfrak{B}_s \Sigma^{1/2}\right) = 0_{\nu\times d},
\]
we have
\[
\begin{aligned}
\rd \mathfrak{X}_t
&= - e^{-t} X_0 \rd t + \rd \big(\!\sqrt{2} e^{-t}\big) \int_0^t e^s \Psi^{1/2} \rd \mathfrak{B}_s \Sigma^{1/2} + \!\sqrt{2} e^{-t} \rd \left(\int_0^t e^s \Psi^{1/2} \rd \mathfrak{B}_s \Sigma^{1/2}\right) \\[-0.5mm]
&= - e^{-t} X_0 \rd t + (e^{-t} X_0 - \mathfrak{X}_t) \rd t + \!\sqrt{2} \Psi^{1/2} \rd \mathfrak{B}_t \Sigma^{1/2} \\[1mm]
&= - \mathfrak{X}_t \rd t + \!\sqrt{2} \Psi^{1/2} \rd \mathfrak{B}_t \Sigma^{1/2}.
\end{aligned}
\]
Since the drift $X\mapsto -X$ is globally Lipschitz and the diffusion coefficient is constant, standard finite-dimensional SDE theory gives pathwise uniqueness, and hence the displayed strong solution is unique. This completes Step~1.

{\bf Step~2:} Using the well-known identity $\vecc(ABC) = (C^{\top} \otimes A) \vecc(B)$ \citep[e.g.,][Eq.~(2.13)]{MR640865}, note that
\begin{equation}\label{eq:step.2.begin}
\vecc(\mathfrak{X}_t^{\top}) = e^{-t} \vecc(X_0^{\top}) + \!\sqrt{2} e^{-t} \int_0^t e^s (\Psi^{1/2} \otimes \Sigma^{1/2}) \rd \vecc(\mathfrak{B}_s^{\top}).
\end{equation}
Since the components of $(\mathfrak{B}_t)_{t\geq 0}$ are all independent, conditionally on $\{\mathfrak{X}_0=X_0\}$, the right-hand side of \eqref{eq:step.2.begin} is a correlated Gaussian random vector. Furthermore, we clearly have
\[
\EE[\vecc(\mathfrak{X}_t^{\top})\mid \mathfrak{X}_0=X_0] = e^{-t} \vecc(X_0^{\top}),
\]
and, using It\^o's isometry and the Kronecker product identity $(A \otimes B)(C \otimes D) = AC \otimes BD$ \citep[e.g.,][Eq.~(2.11)]{MR640865}, we have
\[
\Var\{\vecc(\mathfrak{X}_t^{\top})\mid \mathfrak{X}_0=X_0\}
= 2 e^{-2t} \int_0^t e^{2s} (\Psi^{1/2} \otimes \Sigma^{1/2}) (\Psi^{1/2} \otimes \Sigma^{1/2})^{\top} \rd s
= 2 e^{-2t} \int_0^t e^{2s} \rd s \, (\Psi \otimes \Sigma)
= (1 - e^{-2t}) (\Psi \otimes \Sigma).
\]
Hence, we find that $\mathfrak{X}_t\mid \{\mathfrak{X}_0=X_0\} \sim \mathcal{N}_{\nu\times d}(e^{-t} X_0, (1 - e^{-2t}) (\Psi \otimes \Sigma))$.

{\bf Step~3:} For any vector $\bb{u} \in \R^{\nu d}$, the characteristic function of $\vecc(\mathfrak{X}_t^{\top})$ is
\[
\begin{aligned}
\widehat{\phi}_t(\bb{u})
= \EE\left[\exp\{\mathrm{i}\,\bb{u}^{\top}\vecc(\mathfrak{X}_t^{\top})\}\right]
= \exp\left\{\mathrm{i}\,\bb{u}^{\top} e^{-t}\vecc(X_0^{\top}) - \frac{1}{2} (1 - e^{-2t})\,\bb{u}^{\top}(\Psi \otimes \Sigma)\bb{u}\right\}.
\end{aligned}
\]
Letting $t \to \infty$ gives
\[
\lim_{t \to \infty} \widehat{\phi}_t(\bb{u}) = \exp\left\{-\frac{1}{2} \, \bb{u}^{\top}(\Psi \otimes \Sigma)\bb{u}\right\}, \qquad \bb{u} \in \R^{\nu d},
\]
which is the characteristic function of $\mathcal{N}_{\nu d}(\bb{0}_{\nu d},\Psi \otimes \Sigma)$. By L\'evy's continuity theorem,
\[
\vecc(\mathfrak{X}_t^{\top}) \rightsquigarrow \mathcal{N}_{\nu d}(0,\Psi \otimes \Sigma), \qquad \text{as } t \to \infty,
\]
and thus $\vecc(\mathfrak{X}_{\infty}^{\top}) \sim \mathcal{N}_{\nu d}(\bb{0}_{\nu d},\Psi \otimes \Sigma)$. Equivalently, $\mathfrak{X}_{\infty} \sim \mathcal{N}_{\nu\times d}(0_{\nu\times d},\Psi \otimes \Sigma)$. This concludes the proof.
\end{proof}

\begin{proof}[\bf Proof of Theorem~\ref{thm:Stein.solutions}]\pdfbookmark[2]{Proof of Theorem \ref{thm:Stein.solutions}}{proof:Stein.solutions}
We adapt, in part, the proof of Proposition~6.1 of \citet{MR3980309} to our matrix setting. In particular, we replace their contraction bound (A.1) for Lipschitz test functions on $\R^d$ by the contraction bound in Lemma~\ref{lem:Walpha.contraction.OU} for $\alpha$-H\"older continuous test functions on $\R^{\nu\times d}$. The case $h\in\smash{\mathrm{Lip}_{p}(\R^{\nu\times d})}$ is treated separately at the end of the proof.

Let $\alpha\in(0,1]$ and $h\in C^{0,\alpha}(\R^{\nu\times d})$ be given. Since $|h(X)|\leq |h(0_{\nu\times d})|+[h]_{\alpha}\|X\|_F^{\alpha}$, and the matrix normal laws appearing below have finite $\alpha$-moments, all expectations involving $h(\mathfrak{X}_t)$ and $h(\mathfrak{X}_{\infty})$ below are finite. Throughout the proof, we write
\[
\widehat{h}(X)=\EE[h(\mathfrak{X}_{\infty})]-h(X),
\]
so that the function defined in \eqref{eq:fh.def.OU.alpha} can be rewritten as
\[
f_h(X)=\int_0^{\infty} (\mathcal{P}^{\mathrm{OU}}_t \widehat{h})(X)\rd t.
\]

Fix $X\in\R^{\nu\times d}$ and let $\mu_t^X=\delta_X \mathcal{P}^{\mathrm{OU}}_t$. By definition of the $\alpha$-H\"older--Kantorovich distance on $\R^{\nu\times d}$ in \eqref{dha}, we have
\[
\big|(\mathcal{P}^{\mathrm{OU}}_t \widehat{h})(X)\big|
= \Big|\int h \rd\mu_t^X - \int h \rd\gamma\Big|
\leq [h]_{\alpha} d_{\mathrm{HK},\alpha}(\mu_t^X,\gamma).
\]
By the contraction bound in Lemma~\ref{lem:Walpha.contraction.OU}, $d_{\mathrm{HK},\alpha}(\mu_t^X,\gamma)\leq e^{-\alpha t} d_{\mathrm{HK},\alpha}(\delta_X,\gamma)$, and thus
\begin{equation}\label{eq:step.1.bound}
\big|(\mathcal{P}^{\mathrm{OU}}_t \widehat{h})(X)\big| \leq [h]_{\alpha} e^{-\alpha t} d_{\mathrm{HK},\alpha}(\delta_X,\gamma).
\end{equation}
Integrating \eqref{eq:step.1.bound} over $t\in(0,\infty)$ yields
\[
|f_h(X)|
\leq \int_0^{\infty} [h]_\alpha e^{-\alpha t} d_{\mathrm{HK},\alpha}(\delta_X,\gamma) \rd t
= \alpha^{-1} [h]_{\alpha} d_{\mathrm{HK},\alpha}(\delta_X,\gamma).
\]

We now give a crude bound for $d_{\mathrm{HK},\alpha}(\delta_X,\gamma)$. Given any $g\in C^{0,\alpha}(\R^{\nu\times d})$ with $[g]_{\alpha}\leq 1$, setting $\widehat{g}(X)=\EE[g(\mathfrak{X}_{\infty})]-g(X)$, Jensen's inequality and the $\alpha$-H\"older property imply
\[
\big|\widehat{g}(X)\big|
= \big|\EE[g(\mathfrak{X}_{\infty}) - g(X)]\big|
\leq \EE\big[|g(X)-g(\mathfrak{X}_{\infty})|\big]
\leq \EE\big[\|X-\mathfrak{X}_{\infty}\|_F^{\alpha}\big].
\]
Taking the supremum over such $g$'s yields
\begin{equation}\label{eq:d.H.alpha.bound}
d_{\mathrm{HK},\alpha}(\delta_X,\gamma)\leq \EE\big[\|X-\mathfrak{X}_{\infty}\|_F^{\alpha}\big] \leq \|X\|_F^{\alpha}+\EE\big[\|\mathfrak{X}_{\infty}\|_F^{\alpha}\big],
\end{equation}
where we used the elementary inequality $(u+v)^{\alpha}\leq u^{\alpha}+v^{\alpha}, ~u,v\geq 0, ~\alpha\in (0,1]$. Hence, by combining \eqref{eq:step.1.bound} and \eqref{eq:d.H.alpha.bound}, we have
\begin{equation}\label{eq:P.t.h.hat.bound}
\big|(\mathcal{P}^{\mathrm{OU}}_t \widehat{h})(X)\big| \leq [h]_{\alpha} e^{-\alpha t} \{\|X\|_F^{\alpha}+\EE\big[\|\mathfrak{X}_{\infty}\|_F^{\alpha}\big]\}.
\end{equation}
In particular, this bound shows that $f_h$ is well defined pointwise, and
\[
|f_h(X)|
\leq \alpha^{-1} [h]_{\alpha} d_{\mathrm{HK},\alpha}(\delta_{X}, \gamma)
\leq \alpha^{-1} [h]_{\alpha} \Big\{\EE\big[\|\mathfrak{X}_{\infty}\|_F^{\alpha}\big] + \|X\|_F^{\alpha}\Big\} < \infty,
\]
proving \eqref{eq:fh.bound.OU.alpha}.

We now verify that $f_h$ belongs to the pointwise domain of $\mathcal{A}^{\mathrm{OU}}$ and satisfies $\smash{\mathcal{A}^{\mathrm{OU}} f_h = -\widehat{h}}$. Fix $s>0$ and $X\in\R^{\nu\times d}$. By the Markov property, the semigroup property, and Fubini's theorem (\eqref{eq:P.t.h.hat.bound} implies $\smash{\int_0^{\infty} \EE[|(\mathcal{P}_t^{\mathrm{OU}}\widehat{h})(\mathfrak{X}_s)| \mid \mathfrak{X}_0=X] \, \rd t < \infty}$ since $\mathfrak{X}_s \mid \{\mathfrak{X}_0=X\}$ is matrix normal by Proposition~\ref{prop:OU.process.distribution}), we have
\[
(\mathcal{P}_s^{\mathrm{OU}} f_h)(X)
= \EE\left[\int_0^{\infty} (\mathcal{P}_t^{\mathrm{OU}}\widehat{h})(\mathfrak{X}_s) \rd t \mid \mathfrak{X}_0=X\right]
= \int_0^{\infty} (\mathcal{P}_{t+s}^{\mathrm{OU}}\widehat{h})(X) \rd t.
\]
Changing variables $u=t+s$ gives
\[
(\mathcal{P}_s^{\mathrm{OU}} f_h)(X)
= \int_s^{\infty} (\mathcal{P}_u^{\mathrm{OU}}\widehat{h})(X) \rd u
= f_h(X) - \int_0^s (\mathcal{P}_u^{\mathrm{OU}}\widehat{h})(X) \rd u.
\]
Therefore,
\begin{equation}\label{eq:continuity}
\frac{(\mathcal{P}_s^{\mathrm{OU}} f_h)(X)-f_h(X)}{s}
= -\frac{1}{s} \int_0^s (\mathcal{P}_u^{\mathrm{OU}}\widehat{h})(X) \rd u.
\end{equation}
The map $u\mapsto (\mathcal{P}_u^{\mathrm{OU}}\widehat{h})(X)$ is continuous at $u=0$. Indeed, if $\mathfrak{X}_u^X$ denotes the process started from $X$, then Proposition~\ref{prop:OU.process.distribution} gives $\mathfrak{X}_u^X\to X$ in $L^\alpha$ as $u\downarrow 0$, and
\[
\big|(\mathcal{P}_u^{\mathrm{OU}}\widehat{h})(X)-\widehat{h}(X)\big|
\leq [h]_\alpha \EE\big[\|\mathfrak{X}_u^X-X\|_F^\alpha\big]\to 0,
\]
where $\smash{(\mathcal{P}_0^{\mathrm{OU}}\widehat{h})(X)=\widehat{h}(X)}$. Therefore, letting $s\downarrow 0$ in \eqref{eq:continuity} yields
\begin{equation}\label{eq:again}
\begin{aligned}
\mathcal{A}^{\mathrm{OU}} f_h(X)
= \lim_{s\downarrow 0} \, -\frac{1}{s} \int_0^s (\mathcal{P}_u^{\mathrm{OU}}\widehat{h})(X) \rd u
= -\widehat{h}(X)
= h(X) - \EE[h(\mathfrak{X}_{\infty})],
\end{aligned}
\end{equation}
which is precisely \eqref{eq:Stein.equation.normal}. This proves the asserted statement and bound in the case $h\in C^{0,\alpha}(\R^{\nu\times d})$.

We now prove the asserted statement for test functions in $\smash{\mathrm{Lip}_{p}(\R^{\nu\times d})}$. Since $\mathrm{Lip}_{0}(\R^{\nu\times d}) = C^{0,1}(\R^{\nu\times d})$, the case $p = 0$ follows from the preceding argument with $\alpha = 1$. Let $p\in\N$ and $h\in\smash{\mathrm{Lip}_{p}(\R^{\nu\times d})}$ be given. We first record a simple consequence of the definition of $\smash{\mathrm{Lip}_{p}(\R^{\nu\times d})}$: there exists a finite constant $C_h$ such that, for all $A,B\in\R^{\nu\times d}$,
\begin{equation}\label{eq:Lip.p.increment.bound}
|h(A) - h(B)| \leq C_h \, \big(1 + \|A\|_F^p + \|B\|_F^p\big) \, \|A - B\|_F.
\end{equation}
Indeed, the Lipschitz property of the order-$p$ partial derivatives implies that these derivatives have at most linear growth. Repeatedly integrating lower-order partial derivatives along coordinate line segments shows that $\|\nabla h(Y)\|_F\leq C_h(1 + \|Y\|_F^p)$, possibly after increasing $C_h$. Applying the fundamental theorem of calculus to $r\mapsto h(B + r(A - B))$ gives \eqref{eq:Lip.p.increment.bound}. In particular, $h$ has at most polynomial growth of degree $p+1$, and hence $\EE[|h(\mathfrak{X}_{\infty})|]<\infty$. We again set $\smash{\widehat{h}(X)=\EE[h(\mathfrak{X}_{\infty})]-h(X)}$.

Let $\mathfrak{X}_{\infty}\sim \mathcal{N}_{\nu\times d}(0_{\nu\times d}, \Psi\otimes \Sigma)$, and fix $X\in\R^{\nu\times d}$. For $t\geq 0$, define the coupling
\[
\mathfrak{N}_t^X = e^{-t} X + \!\sqrt{1 - e^{-2t}} \, \mathfrak{X}_{\infty}.
\]
Then $\mathfrak{N}_t^X\sim\delta_X \mathcal{P}^{\mathrm{OU}}_t$, and therefore
\[
(\mathcal{P}^{\mathrm{OU}}_t\widehat{h})(X)
= \EE[h(\mathfrak{X}_{\infty})] - \EE[h(\mathfrak{N}_t^X)]
= \EE[h(\mathfrak{X}_{\infty}) - h(\mathfrak{N}_t^X)].
\]
Moreover,
\[
\|\mathfrak{N}_t^X - \mathfrak{X}_{\infty}\|_F
\leq e^{-t} \|X\|_F + \big(1 - \!\sqrt{1 - e^{-2t}}\big)\|\mathfrak{X}_{\infty}\|_F
\leq e^{-t}(\|X\|_F + \|\mathfrak{X}_{\infty}\|_F),
\]
because $1 - \!\sqrt{1 - e^{-2t}} = e^{-2t}/(1 + \!\sqrt{1 - e^{-2t}}) \leq e^{-2t} \leq e^{-t}$. Using \eqref{eq:Lip.p.increment.bound} and the bound $\|\mathfrak{N}_t^X\|_F\leq \|X\|_F + \|\mathfrak{X}_{\infty}\|_F$, we obtain
\[
\begin{aligned}
\big|(\mathcal{P}^{\mathrm{OU}}_t\widehat{h})(X)\big|
&\leq C_h\, \EE\Big[\big(1 + \|\mathfrak{X}_{\infty}\|_F^p + \|\mathfrak{N}_t^X\|_F^p\big) \, \|\mathfrak{X}_{\infty} - \mathfrak{N}_t^X\|_F\Big] \\
&\leq C_h e^{-t} \, \EE\Big[\big(1 + \|\mathfrak{X}_{\infty}\|_F^p + (\|X\|_F + \|\mathfrak{X}_{\infty}\|_F)^p\big) \, (\|X\|_F + \|\mathfrak{X}_{\infty}\|_F)\Big].
\end{aligned}
\]
The expectation on the right-hand side is finite and independent of $t$. Thus $\smash{\int_0^{\infty} |(\mathcal{P}^{\mathrm{OU}}_t\widehat{h})(X)| \rd t < \infty}$, which proves that the integral in \eqref{eq:fh.def.OU.alpha} is absolutely convergent for every fixed $X\in\R^{\nu\times d}$.

It remains only to verify the Stein identity. The preceding estimate, with $X$ replaced by an arbitrary $Y$, implies that there exists a finite constant $C$ such that, for all $Y\in\R^{\nu\times d}$ and $t\geq 0$,
\begin{equation}\label{eq:Lip.p.semigroup.bound}
\big|(\mathcal{P}^{\mathrm{OU}}_t\widehat{h})(Y)\big|
\leq C e^{-t}\big(1 + \|Y\|_F^{p+1}\big).
\end{equation}
Since $\mathfrak{X}_s\mid\{\mathfrak{X}_0=X\}$ has finite moments of all orders, \eqref{eq:Lip.p.semigroup.bound} justifies the same Fubini and semigroup argument used above. Hence, for every $s>0$, \eqref{eq:continuity} holds as before. Moreover, if $\mathfrak{X}_u^X$ denotes the process started from $X$, then $\mathfrak{X}_u^X\to X$ in $L^{p+1}$ as $u\downarrow 0$. By \eqref{eq:Lip.p.increment.bound} and H\"older's inequality,
\[
\begin{aligned}
\big|(\mathcal{P}_u^{\mathrm{OU}}\widehat{h})(X)-\widehat{h}(X)\big|
&\leq \EE\big[|h(\mathfrak{X}_u^X)-h(X)|\big] \\
&\leq C_h \Big(\EE\big[\big(1+\|\mathfrak{X}_u^X\|_F^p+\|X\|_F^p\big)^{(p+1)/p}\big]\Big)^{p/(p+1)} \Big(\EE\big[\|\mathfrak{X}_u^X-X\|_F^{p+1}\big]\Big)^{1/(p+1)} \to 0, \qquad u\downarrow 0.
\end{aligned}
\]
Consequently, \eqref{eq:again} holds as before, which is the matrix normal Stein equation \eqref{eq:Stein.equation.normal}. This proves the asserted statement in the case $h\in\smash{\mathrm{Lip}_{p}(\R^{\nu\times d})}$ and completes the proof.
\end{proof}

\begin{proof}[\bf Proof of Theorem~\ref{thm:smoothness.estimates}]\pdfbookmark[2]{Proof of Theorem \ref{thm:smoothness.estimates}}{proof:smoothness.estimates}
Let $m\in\N$, and write $\overline{\nabla} \equiv \prod_{\ell=1}^m \nabla_{i_{\ell} j_{\ell}}$ for short. Assume first that $h\in \mathrm{Lip}_{m-1}(\R^{\nu\times d})$. By dominated convergence and the semigroup representation of $f_h$, we have, for all $X\in \R^{\nu\times d}$,
\[
\big|\overline{\nabla} f_h(X)\big|
= \left|\int_0^{\infty} e^{-mt} \, \EE\left[\big(\overline{\nabla} h\big)\big(e^{-t} X + \!\sqrt{1 - e^{-2t}} \, \Psi^{1/2} \mathfrak{Z} \Sigma^{1/2}\big)\right] \rd t\right|
\leq \int_0^{\infty} e^{-mt} \rd t \, \|\overline{\nabla} h\|_{\infty}
= \frac{1}{m} \|\overline{\nabla} h\|_{\infty},
\]
which proves \eqref{eq:Stein.bound.1}.

Next, assume that $h\in C^{0,\alpha}(\R^{\nu\times d})$ is bounded for some $\alpha\in(0,1]$. For $t\geq 0$, $X\in \R^{\nu\times d}$, and $Z\in \R^{\nu\times d}$, define
\[
\lambda_{t,X}(Z) = e^{-t} X + \!\sqrt{1 - e^{-2t}} \, \Psi^{1/2} Z \Sigma^{1/2}.
\]
Let $\phi \equiv \phi_{0_{\nu\times d}, I_{\nu}, I_d}$ denote the density of the $\mathcal{N}_{\nu\times d}(0_{\nu\times d},I_{\nu}\otimes I_d)$ distribution, namely
\[
\phi(Z) = \etr(-Z^{\top} Z / 2)/(2\pi)^{\nu d/2}.
\]
By the semigroup representation of $f_h$ in \eqref{eq:fh.def.OU.alpha}, we have
\[
f_h(X)
= - \int_0^{\infty} \left\{\int_{\R^{\nu\times d}} h(\lambda_{t,X}(Z)) \, \phi(Z) \rd Z - \EE\big[h(\Psi^{1/2}\mathfrak{Z}\Sigma^{1/2})\big]\right\} \rd t.
\]
For each fixed $t>0$, make the change of variable $Y=\lambda_{t,X}(Z)$. Then
\[
\lambda_{t,X}^{-1}(Y)
= \frac{1}{\!\sqrt{1 - e^{-2t}}} \Psi^{-1/2} Y \Sigma^{-1/2}
- \frac{e^{-t}}{\!\sqrt{1 - e^{-2t}}} \Psi^{-1/2} X \Sigma^{-1/2},
\]
and the Jacobian determinant of the transformation $Z=\lambda_{t,X}^{-1}(Y)$ is $|\Psi|^{-d/2} |\Sigma|^{-\nu/2}/(1 - e^{-2t})^{\nu d/2}$. Therefore,
\[
\int_{\R^{\nu\times d}} h(\lambda_{t,X}(Z)) \, \phi(Z) \rd Z = \int_{\R^{\nu\times d}} h(Y) \, \phi(\lambda_{t,X}^{-1}(Y))
\frac{|\Psi|^{-d/2} |\Sigma|^{-\nu/2}}{(1 - e^{-2t})^{\nu d/2}} \rd Y.
\]
Since $h$ is bounded, dominated convergence yields, for every $(i,j)\in [\nu]\times [d]$,
\[
\frac{\partial f_h(X)}{\partial X_{ij}}
= - \int_0^{\infty} \frac{\partial}{\partial X_{ij}}
\int_{\R^{\nu\times d}} h(Y) \, \phi(\lambda_{t,X}^{-1}(Y))
\frac{|\Psi|^{-d/2} |\Sigma|^{-\nu/2}}{(1 - e^{-2t})^{\nu d/2}} \rd Y \rd t.
\]
Now, only $\lambda_{t,X}^{-1}(Y)$ depends on $X$, and
\[
\frac{\partial}{\partial X_{ij}} \tr\{-\lambda_{t,X}^{-1}(Y)^{\top}\lambda_{t,X}^{-1}(Y)/2\}
= \frac{e^{-t}}{\!\sqrt{1 - e^{-2t}}} \, (\Psi^{-1/2}\lambda_{t,X}^{-1}(Y)\Sigma^{-1/2})_{ij}.
\]
It follows that
\[
\frac{\partial}{\partial X_{ij}} \phi(\lambda_{t,X}^{-1}(Y))
= \frac{e^{-t}}{\!\sqrt{1 - e^{-2t}}} \, (\Psi^{-1/2}\lambda_{t,X}^{-1}(Y)\Sigma^{-1/2})_{ij}
\, \phi(\lambda_{t,X}^{-1}(Y)),
\]
and thus
\[
\begin{aligned}
\frac{\partial f_h(X)}{\partial X_{ij}}
&= - \int_0^{\infty} \frac{e^{-t}}{\!\sqrt{1 - e^{-2t}}}
\int_{\R^{\nu\times d}} (\Psi^{-1/2}\lambda_{t,X}^{-1}(Y)\Sigma^{-1/2})_{ij} \, h(Y) \, \phi(\lambda_{t,X}^{-1}(Y))
\frac{|\Psi|^{-d/2} |\Sigma|^{-\nu/2}}{(1 - e^{-2t})^{\nu d/2}} \rd Y \rd t.
\end{aligned}
\]
Reverting the change of variable $Y=\lambda_{t,X}(Z)$ gives
\begin{equation}\label{eq:m.equal.1.case}
\frac{\partial f_h(X)}{\partial X_{ij}}
= - \int_0^{\infty} \frac{e^{-t}}{\!\sqrt{1 - e^{-2t}}} \, \EE\big[(\Psi^{-1/2}\mathfrak{Z}\Sigma^{-1/2})_{ij} \, h(\lambda_{t,X}(\mathfrak{Z}))\big] \rd t.
\end{equation}
Since $\EE[(\Psi^{-1/2}\mathfrak{Z}\Sigma^{-1/2})_{ij}]=0$, we may rewrite this as
\[
\frac{\partial f_h(X)}{\partial X_{ij}}
= - \int_0^{\infty} \frac{e^{-t}}{\!\sqrt{1 - e^{-2t}}} \, \EE\Big[(\Psi^{-1/2}\mathfrak{Z}\Sigma^{-1/2})_{ij}
\big\{h(\lambda_{t,X}(\mathfrak{Z})) - \EE[h(\Psi^{1/2}\mathfrak{Z}\Sigma^{1/2})]\big\}\Big] \rd t.
\]
Taking absolute values yields, for all $X\in \R^{\nu\times d}$,
\[
\left|\nabla_{ij} f_h(X)\right|
\leq \big\|h - \EE[h(\Psi^{1/2}\mathfrak{Z}\Sigma^{1/2})]\big\|_{\infty}
\int_0^{\infty} \frac{e^{-t}}{\!\sqrt{1 - e^{-2t}}} \rd t ~ \EE\big[|(\Psi^{-1/2}\mathfrak{Z}\Sigma^{-1/2})_{ij}|\big].
\]
Now,
\[
\int_0^{\infty} \frac{e^{-t}}{\!\sqrt{1 - e^{-2t}}} \rd t
= \frac{1}{2}\int_0^1 u^{-1/2} (1-u)^{-1/2} \rd u
= \frac{\pi}{2},
\]
and, since $(\Psi^{-1/2}\mathfrak{Z}\Sigma^{-1/2})_{ij}$ is centered normal with variance $(\Psi^{-1})_{ii}(\Sigma^{-1})_{\!j\hspace{-0.3mm}j}$,
\begin{equation}\label{eq:bound.2.end}
\EE\big[|(\Psi^{-1/2}\mathfrak{Z}\Sigma^{-1/2})_{ij}|\big]
= \!\sqrt{\frac{2}{\pi}} \!\sqrt{(\Psi^{-1})_{ii}(\Sigma^{-1})_{\!j\hspace{-0.3mm}j}}.
\end{equation}
Putting the last three equations together proves \eqref{eq:Stein.bound.2}.

We now prove \eqref{eq:Stein.bound.3}. Fix $k\in[m]$, and write $\smash{\overline{\nabla}_{-k\,}\equiv \smash{\prod_{\ell=1,\ell\neq k}^m \nabla_{i_{\ell} j_{\ell}}}}$ for short. Let $m\in\N\setminus\{1\}$ be given and assume that $\smash{h\in \mathrm{Lip}_{m-2}(\R^{\nu\times d})}$. Differentiating the semigroup representation \eqref{eq:fh.def.OU.alpha} $m-1$ times under the integral sign yields
\[
\overline{\nabla}_{-k\,} f_h(X)
= - \int_0^{\infty} e^{-(m-1)t} \, \EE\left[\overline{\nabla}_{-k\,} h\big(\lambda_{t,X}(\mathfrak{Z})\big)\right] \rd t.
\]
Applying the preceding argument that yielded \eqref{eq:m.equal.1.case} to the bounded test function $\overline{\nabla}_{-k\,} h$ gives
\[
\overline{\nabla} f_h(X)
= - \int_0^{\infty} \frac{e^{-mt}}{\!\sqrt{1 - e^{-2t}}} \, \EE\left[(\Psi^{-1/2}\mathfrak{Z}\Sigma^{-1/2})_{i_k,j_k}
\, \overline{\nabla}_{-k\,} h\big(\lambda_{t,X}(\mathfrak{Z})\big)\right] \rd t.
\]
Hence, for all $X\in \R^{\nu\times d}$,
\[
\big|\overline{\nabla} f_h(X)\big|
\leq \big\|\overline{\nabla}_{-k\,} h\big\|_{\infty}
\int_0^{\infty} \frac{e^{-mt}}{\!\sqrt{1 - e^{-2t}}} \rd t ~ \EE\big[|(\Psi^{-1/2}\mathfrak{Z}\Sigma^{-1/2})_{i_k,j_k}|\big].
\]
Taking the minimum over $k\in[m]$, we obtain
\[
\big|\overline{\nabla} f_h(X)\big|
\leq \int_0^{\infty} \frac{e^{-mt}}{\!\sqrt{1 - e^{-2t}}} \rd t \,
\min_{1\leq k \leq m} \left\{\EE\big[|(\Psi^{-1/2}\mathfrak{Z}\Sigma^{-1/2})_{i_k j_k}|\big]
\left\|\overline{\nabla}_{-k\,} h\right\|_{\infty}\right\}.
\]
Using the change of variable $u = 1 - e^{-2t}$ ($2 \rd t = (1-u)^{-1} \rd u$), note that
\begin{equation}\label{eq:Beta.integral}
\int_0^{\infty} \frac{e^{-mt}}{\!\sqrt{1 - e^{-2t}}} \rd t
= \frac{1}{2} \int_0^1 u^{-1/2} (1-u)^{m/2 - 1} \rd u
= \frac{\!\sqrt{\pi}\, \Gamma(m/2)}{2\, \Gamma(m/2+1/2)}.
\end{equation}
Moreover, as in \eqref{eq:bound.2.end}, we have
\[
\EE\big[|(\Psi^{-1/2}\mathfrak{Z}\Sigma^{-1/2})_{i_k j_k}|\big]
= \!\sqrt{\frac{2}{\pi}} \!\sqrt{(\Psi^{-1})_{i_k i_k} (\Sigma^{-1})_{j_k j_k}}.
\]
Putting the last three equations together proves \eqref{eq:Stein.bound.3}.

We now prove \eqref{eq:Stein.bound.4}. Assume that $h\in C_b^m(\R^{\nu\times d})$ for some $m\in\N$. Recall the semigroup representation
\[
f_h(X)
= - \int_0^{\infty} \left\{\EE\big[h\big(\lambda_{t,X}(\mathfrak{Z})\big)\big] - \EE\big[h(\Psi^{1/2}\mathfrak{Z}\Sigma^{1/2})\big]\right\} \rd t,
\]
where $\lambda_{t,X}(Z)=e^{-t} X + \!\sqrt{1 - e^{-2t}} \, \Psi^{1/2} Z \Sigma^{1/2}$. Since the second expectation is constant in $X$, its Fr\'echet derivatives vanish. Fix $t > 0$ and $Z\in \R^{\nu\times d}$. By the chain rule and multilinearity of Fr\'echet derivatives, we have for all $X\in \R^{\nu\times d}$ and $U_1,\ldots,U_m\in \R^{\nu\times d}$,
\[
D^m \{h(\lambda_{t,\bigcdot}(Z))\}(X)[U_1,\ldots,U_m]
= D^m h(\lambda_{t,X}(Z))[e^{-t}U_1,\ldots,e^{-t}U_m]
= e^{-mt} D^m h(\lambda_{t,X}(Z))[U_1,\ldots,U_m].
\]
Since $h\in C_b^m(\R^{\nu\times d})$, we have $\mathcal{M}_m(h)<\infty$, and thus
\[
\left|D^m \{h(\lambda_{t,\bigcdot}(Z))\}(X)[U_1,\ldots,U_m]\right|
\leq e^{-mt} \, \mathcal{M}_m(h) \prod_{\ell=1}^m \|U_\ell\|_F.
\]
In particular, the right-hand side is integrable in $t$, so, by the dominated convergence theorem, we can differentiate under the expectation and integral signs to obtain
\[
D^m f_h(X)[U_1,\ldots,U_m]
= - \int_0^{\infty} e^{-mt} \, \EE\left[D^m h\big(\lambda_{t,X}(\mathfrak{Z})\big)[U_1,\ldots,U_m]\right] \rd t.
\]
Thus, for $\|U_1\|_F = \cdots = \|U_m\|_F = 1$,
\[
\big|D^m f_h(X)[U_1,\ldots,U_m]\big|
\leq \int_0^{\infty} e^{-mt} \rd t \, \mathcal{M}_m(h)
= \frac{1}{m}\mathcal{M}_m(h).
\]
Taking the suprema over $X\in\R^{\nu\times d}$ and over $\|U_1\|_F = \cdots = \|U_m\|_F = 1$ yields \eqref{eq:Stein.bound.4}.

We end by proving \eqref{eq:Stein.bound.6}; the proof of \eqref{eq:Stein.bound.5} under the bounded H\"older assumption is obtained by the same argument with $m=1$ if $D^0 h$ is understood as $h$.
Fix $m\in\N\setminus\{1\}$ and $k\in[m]$, and let $U_1,\ldots,U_m\in \R^{\nu\times d}$ with $\|U_1\|_F = \cdots = \|U_m\|_F = 1$.
From the semigroup representation \eqref{eq:fh.def.OU.alpha} and differentiation under the integral sign for the $m-1$ derivatives that fall on $h$, we have
\[
D^{m-1} f_h(X)[U_1,\ldots,U_{k-1},U_{k+1},\ldots,U_m]
= - \int_0^{\infty} e^{-(m-1)t}\, \EE\Big[D^{m-1} h\big(\lambda_{t,X}(\mathfrak{Z})\big)[U_1,\ldots,U_{k-1},U_{k+1},\ldots,U_m]\Big] \rd t.
\]
Now apply the preceding $m=1$ argument to the bounded test function
\[
g(Y) = D^{m-1} h(Y)[U_1,\ldots,U_{k-1},U_{k+1},\ldots,U_m], \qquad Y\in \R^{\nu\times d}.
\]
This yields
\[
D^{m} f_h(X)[U_1,\ldots,U_m]
= - \int_0^{\infty} \frac{e^{-mt}}{\!\sqrt{1 - e^{-2t}}} \, \EE\Big[\langle \Psi^{-1/2}\mathfrak{Z}\Sigma^{-1/2},U_k\rangle_F D^{m-1} h\big(\lambda_{t,X}(\mathfrak{Z})\big)[U_1,\ldots,U_{k-1},U_{k+1},\ldots,U_m]\Big] \rd t.
\]
Taking absolute values and the suprema that define $\mathcal{M}_m(f_h)$ gives
\[
\mathcal{M}_m(f_h)
\leq \int_0^{\infty} \frac{e^{-mt}}{\!\sqrt{1 - e^{-2t}}} \rd t \,
\left(\sup_{\|U\|_F=1}\EE\big[|\langle \Psi^{-1/2}\mathfrak{Z}\Sigma^{-1/2},U\rangle_F|\big]\right)
\, \mathcal{M}_{m-1}(h).
\]
Since $\langle \Psi^{-1/2}\mathfrak{Z}\Sigma^{-1/2},U\rangle_F \sim \mathcal{N}(0,\|\Psi^{-1/2}U\Sigma^{-1/2}\|_F^2)$, we have
\[
\EE\big[|\langle \Psi^{-1/2}\mathfrak{Z}\Sigma^{-1/2},U\rangle_F|\big]
= \!\sqrt{\frac{2}{\pi}}\,\|\Psi^{-1/2}U\Sigma^{-1/2}\|_F
\leq \!\sqrt{\frac{2}{\pi}}\,\|\Psi^{-1/2}\|_2 \, \|U\|_F \, \|\Sigma^{-1/2}\|_2,
\]
where the last inequality is due to \eqref{eq:Frobenius.spectral.submult}. Therefore, the supremum above is bounded by $\!\sqrt{2/\pi}\,\|\Psi^{-1/2}\|_2\,\|\Sigma^{-1/2}\|_2$. Using \eqref{eq:Beta.integral}, we conclude that
\[
\mathcal{M}_m(f_h)
\leq \frac{\Gamma(m/2)}{\!\sqrt{2}\, \Gamma(m/2 + 1/2)}\, \|\Psi^{-1/2}\|_2\, \|\Sigma^{-1/2}\|_2 \, \mathcal{M}_{m-1}(h),
\]
which is \eqref{eq:Stein.bound.6}. This completes the proof.
\end{proof}

\begin{proof}[\bf Proof of Proposition~\ref{prop:generalization.Gaunt.2013.Example.2.6}]\pdfbookmark[2]{Proof of Proposition \ref{prop:generalization.Gaunt.2013.Example.2.6}}{proof:generalization.Gaunt.2013.Example.2.6}
Recall that $\mathfrak{S} = n^{-1/2} \smash{\sum_{k=1}^n \mathfrak{X}_k}$, and define the leave-one-out version, namely
\[
\mathfrak{S}_k = \mathfrak{S} - \frac{1}{\!\sqrt{n}} \mathfrak{X}_k, \qquad k\in [n].
\]
Note that $\mathfrak{S}_k$ and $\mathfrak{X}_k$ are independent and that $\mathfrak{S} = \mathfrak{S}_k + n^{-1/2} \mathfrak{X}_k$.

Fix $f\in \mathrm{Lip}_{2}(\R^{\nu\times d})$ such that the expectations appearing below are finite. In the Taylor expansions below, the quantities
$\|\nabla_{i_3j_3}\nabla_{i_2j_2}\nabla_{i_1j_1}f\|_{\infty}$ are understood as the Lipschitz constants of the second-order partial derivatives, in the sense described before Theorem~\ref{thm:Stein.solutions}. We begin by rewriting the drift term:
\[
\EE[\tr\{\mathfrak{S}^{\top} \nabla f(\mathfrak{S})\}]
= \sum_{i_1=1}^{\nu} \sum_{j_1=1}^d \EE[\mathfrak{S}_{i_1j_1} \nabla_{i_1j_1} f(\mathfrak{S})]
= \frac{1}{\!\sqrt{n}} \sum_{k=1}^n \sum_{i_1=1}^{\nu} \sum_{j_1=1}^d \EE[(\mathfrak{X}_k)_{i_1j_1} \nabla_{i_1j_1} f(\mathfrak{S})].
\]
Since $\mathfrak{S} = \mathfrak{S}_k + n^{-1/2}\mathfrak{X}_k$, a Taylor expansion of the map $U \mapsto \nabla_{i_1j_1} f(U)$ around $\mathfrak{S}_k$ yields
\[
\EE[\tr\{\mathfrak{S}^{\top} \nabla f(\mathfrak{S})\}]
= \frac{1}{\!\sqrt{n}} \sum_{k=1}^n \sum_{i_1=1}^{\nu} \sum_{j_1=1}^d \EE[(\mathfrak{X}_k)_{i_1j_1} \nabla_{i_1j_1} f(\mathfrak{S}_k)] + \frac{1}{n} \sum_{k=1}^n \sum_{i_1,i_2=1}^{\nu} \sum_{j_1,j_2=1}^d \EE[(\mathfrak{X}_k)_{i_1j_1} (\mathfrak{X}_k)_{i_2j_2} \nabla_{i_2j_2} \nabla_{i_1j_1} f(\mathfrak{S}_k)] + R_1(f),
\]
where the remainder satisfies
\[
|R_1(f)| \leq \frac{1}{2 \!\sqrt{n}} \max_{\substack{i_1,i_2,i_3\in [\nu] \\ j_1,j_2,j_3\in [d]}} \|\nabla_{i_3j_3} \nabla_{i_2j_2} \nabla_{i_1j_1} f\|_{\infty} \sum_{i_1,i_2,i_3=1}^{\nu} \sum_{j_1,j_2,j_3=1}^d \EE[|\mathfrak{X}_{i_1j_1} \mathfrak{X}_{i_2j_2} \mathfrak{X}_{i_3j_3}|].
\]
By independence of $\mathfrak{S}_k$ and $\mathfrak{X}_k$ together with $\EE[(\mathfrak{X}_k)_{i_1j_1}] = 0$, we have
\[
\EE[(\mathfrak{X}_k)_{i_1j_1} \nabla_{i_1j_1} f(\mathfrak{S}_k)]
= \EE[(\mathfrak{X}_k)_{i_1j_1}] \EE[\nabla_{i_1j_1} f(\mathfrak{S}_k)] = 0.
\]
Moreover, for all $(i_1,j_1),(i_2,j_2)\in [\nu]\times [d]$,
\[
\EE[(\mathfrak{X}_k)_{i_1j_1} (\mathfrak{X}_k)_{i_2j_2}]
= (\Psi \otimes \Sigma)_{(i_1 - 1) d + j_1, (i_2 - 1) d + j_2}
= \Psi_{i_1i_2} \Sigma_{j_1j_2}.
\]
Inserting this back into the drift identity yields
\[
\EE[\tr\{\mathfrak{S}^{\top} \nabla f(\mathfrak{S})\}]
= \frac{1}{n} \sum_{k=1}^n \sum_{i_1,i_2=1}^{\nu} \sum_{j_1,j_2=1}^d \Psi_{i_1i_2} \Sigma_{j_1j_2} \EE[\nabla_{i_2j_2} \nabla_{i_1j_1} f(\mathfrak{S}_k)] + R_1(f).
\]

Apply a first-order Taylor expansion to the map $U\mapsto \nabla_{i_2j_2}\nabla_{i_1j_1} f(U)$ to get
\[
\big|\EE[\nabla_{i_2j_2} \nabla_{i_1j_1} f(\mathfrak{S})] - \EE[\nabla_{i_2j_2} \nabla_{i_1j_1} f(\mathfrak{S}_k)]\big|
\leq \frac{1}{\!\sqrt{n}} \sum_{i_3=1}^{\nu} \sum_{j_3=1}^d \|\nabla_{i_3j_3} \nabla_{i_2j_2} \nabla_{i_1j_1} f\|_{\infty} \, \EE[|(\mathfrak{X}_k)_{i_3j_3}|].
\]
Using this bound and the triangle inequality, we obtain
\begin{equation}\label{eq:full.expansion}
\begin{aligned}
\EE[\tr\{\mathfrak{S}^{\top} \nabla f(\mathfrak{S})\}]
&= \sum_{i_1,i_2=1}^{\nu} \sum_{j_1,j_2=1}^d \Psi_{i_1i_2} \Sigma_{j_1j_2} \EE[\nabla_{i_2j_2} \nabla_{i_1j_1} f(\mathfrak{S})] + R_1(f) + R_2(f), \\
&= \EE[\tr\{\Sigma \nabla^{\top} \Psi \nabla f(\mathfrak{S})\}] + R_1(f) + R_2(f),
\end{aligned}
\end{equation}
with
\[
|R_2(f)| \leq \frac{1}{\!\sqrt{n}} \max_{\substack{i_1,i_2,i_3\in [\nu] \\ j_1,j_2,j_3\in [d]}} \|\nabla_{i_3j_3} \nabla_{i_2j_2} \nabla_{i_1j_1} f\|_{\infty} \sum_{i_1,i_2,i_3=1}^{\nu} \sum_{j_1,j_2,j_3=1}^d |\Psi_{i_1i_2} \Sigma_{j_1j_2}| \, \EE[|\mathfrak{X}_{i_3j_3}|].
\]
Now take $f = f_h$, where $f_h$ is the semigroup solution associated with $h$. This is legitimate for the two classes of test functions used below: $\mathcal{H}_3\subseteq \mathrm{Lip}_{2}(\R^{\nu\times d})$ and $\mathcal{H}_2\subseteq \mathrm{Lip}_{1}(\R^{\nu\times d})$. Hence the $\mathrm{Lip}_{p}$ part of Theorem~\ref{thm:Stein.solutions} gives the Stein identity, while Theorem~\ref{thm:smoothness.estimates} gives the third-order derivative bounds on $f_h$ needed in the Taylor expansion. By \eqref{eq:Stein.equation.normal}, we have
\[
\tr\{\Sigma \nabla^{\top} \Psi \nabla f_h(X)\} - \tr\{X^{\top}\nabla f_h(X)\}
= \mathcal{A}^{\mathrm{OU}} f_h(X)
= h(X) - \EE[h(\mathfrak{Z})].
\]
Evaluating at $X=\mathfrak{S}$ and using \eqref{eq:full.expansion} in conjunction with the triangle inequality implies
\[
\big|\EE[h(\mathfrak{S})] - \EE[h(\mathfrak{Z})]\big| \leq |R_1(f_h)| + |R_2(f_h)|.
\]
By inequality \eqref{eq:Stein.bound.1} of Theorem~\ref{thm:smoothness.estimates} with $m=3$ and $h\in \mathcal{H}_3$,
\[
\max_{\substack{i_1,i_2,i_3\in [\nu] \\ j_1,j_2,j_3\in [d]}} \|\nabla_{i_3j_3}\nabla_{i_2j_2}\nabla_{i_1j_1} f_h\|_{\infty}
\leq \frac{1}{3} \max_{\substack{i_1,i_2,i_3\in [\nu] \\ j_1,j_2,j_3\in [d]}} \|\nabla_{i_3j_3}\nabla_{i_2j_2}\nabla_{i_1j_1} h\|_{\infty}
\leq\frac{1}{3}.
\]
The claimed bound \eqref{cltbd1} now follows. Similarly, by inequality \eqref{eq:Stein.bound.3} with $m=3$ and $h\in\mathcal{H}_2$,
\[
\max_{\substack{i_1,i_2,i_3\in [\nu] \\ j_1,j_2,j_3\in [d]}} \|\nabla_{i_3j_3}\nabla_{i_2j_2}\nabla_{i_1j_1} f_h\|_{\infty}
\leq \frac{\Gamma(3/2)}{\!\sqrt{2} \, \Gamma(2)} \, \max_{i\in [\nu], j\in [d]} \!\sqrt{(\Psi^{-1})_{ii} (\Sigma^{-1})_{\!j\hspace{-0.3mm}j}}
= \frac{\!\sqrt{2\pi}}{4} \, \max_{i\in [\nu], j\in [d]} \!\sqrt{(\Psi^{-1})_{ii} (\Sigma^{-1})_{\!j\hspace{-0.3mm}j}},
\]
and the claimed bound \eqref{cltbd2} now follows.
\end{proof}

\begin{proof}[\bf Proof of Proposition~\ref{prop:T.vs.normal}]\pdfbookmark[2]{Proof of Proposition \ref{prop:T.vs.normal}}{proof:T.vs.normal}
Let $\mathfrak{T}_n\sim \mathcal{T}_{\nu\times d}(n,0_{\nu\times d},\Psi\otimes\Sigma)$ and $\mathfrak{Z}\sim \mathcal{N}_{\nu\times d}(0_{\nu\times d}, \Psi \otimes \Sigma)$. Assume for now that $n > 3$, which is needed for the moments $\EE[\|\Psi^{-1/2} \mathfrak{T}_n \Sigma^{-1/2}\|_F]$ and $\EE[\|\Psi^{-1/2} \mathfrak{T}_n \Sigma^{-1/2}\|_F^3]$ to be finite later in the proof.

To derive a Stein operator, we exhibit an overdamped Langevin diffusion for which \eqref{eq:matrix.T} is stationary. Define the change of variable, $Y = \Psi^{-1/2} X \Sigma^{-1/2}$, to map $X\in \R^{\nu\times d}$ to an isotropic coordinate $Y\in \R^{\nu\times d}$. In $Y$-coordinates, the target density of $\mathfrak{Y} = \Psi^{-1/2} \mathfrak{T}_n \Sigma^{-1/2} \sim \mathcal{T}_{\nu\times d}(n,0_{\nu\times d},I_{\nu}\otimes I_d)$ is
\[
p_{\mathfrak{Y}}(Y)= \frac{\Gamma_{\nu}((n+\nu+d-1)/2) \big|I_{\nu} + n^{-1} Y Y^{\top}\big|^{-(n+\nu+d-1)/2}}{(\pi n)^{\nu d/2} \Gamma_{\nu}((n + \nu - 1)/2)},
\]
and the corresponding Langevin SDE is
\[
\rd \mathfrak{Y}_t = \nabla_Y \log p_{\mathfrak{Y}}(\mathfrak{Y}_t) \rd t + \!\sqrt{2} \, \rd \mathfrak{B}_t,
\]
where $\nabla_Y \log p_{\mathfrak{Y}}(Y) = -\{(n+\nu+d-1)/2\} (I_{\nu} + n^{-1} Y Y^{\top})^{-1} (2 n^{-1} Y)$, and $\mathfrak{B}_t$ is a $\nu\times d$ matrix of independent standard Brownian motions. Pushing back to $\mathfrak{T}_{n,t}$ via $\smash{\mathfrak{T}_{n,t} = \Psi^{1/2} \mathfrak{Y}_t \Sigma^{1/2}}$ yields
\begin{equation}\label{eq:T.Langevin.SDE}
\rd \mathfrak{T}_{n,t} = b(\mathfrak{T}_{n,t}) \rd t + \!\sqrt{2} \Psi^{1/2} \rd \mathfrak{B}_t \Sigma^{1/2},
\end{equation}
with drift
\[
\begin{aligned}
b(X)
&= \Psi^{1/2} \left\{\left.\nabla_Y \log p_{\mathfrak{Y}}(Y)\right|_{Y = \Psi^{-1/2} X \Sigma^{-1/2}}\right\} \Sigma^{1/2} \\
&= -\frac{n+\nu+d-1}{n} \Psi^{1/2} (I_{\nu} + n^{-1} \Psi^{-1/2} X \Sigma^{-1} X^{\top} \Psi^{-1/2})^{-1} (\Psi^{-1/2} X \Sigma^{-1/2}) \Sigma^{1/2}.
\end{aligned}
\]
A standard Fokker--Planck argument shows that \eqref{eq:matrix.T} is the unique invariant density of \eqref{eq:T.Langevin.SDE}.

By It\^o's formula, the extended generator of the Langevin diffusion \eqref{eq:T.Langevin.SDE} acts on $f\in C^2(\R^{\nu\times d})$ as
\begin{equation}\label{eq:T.generator}
\mathcal{A}^{\mathrm{L}} f(X)
= \tr\{b(X)^{\top} \nabla_X f(X)\} + \tr\{\Sigma \nabla_X^{\top} \Psi \nabla_X f(X)\}.
\end{equation}
Since the process is ergodic with stationary law $\mathcal{T}_{\nu\times d}(n,0_{\nu\times d},\Psi\otimes\Sigma)$, it follows that
\[
\EE\big[\mathcal{A}^{\mathrm{L}} f(\mathfrak{T}_n)\big] = 0, \qquad \forall f\in C_{\mathcal{A}^{\mathrm{L}}}^2(\R^{\nu\times d}),
\]
where
\[
\begin{aligned}
C_{\mathcal{A}^{\mathrm{L}}}^2(\R^{\nu\times d})
&= \Big\{f\in C^2(\R^{\nu\times d}) \, : \, \EE[|\tr\{b(\mathfrak{T}_n)^{\top} \nabla_X f(\mathfrak{T}_n)\}|] < \infty, ~\EE[|\tr\{\Sigma \nabla_X^{\top} \Psi \nabla_X f(\mathfrak{T}_n)\}|] < \infty \Big\}.
\end{aligned}
\]
This Stein identity parallels that of the matrix normal distribution in Corollary~\ref{cor:Stein.normal.step.1} and provides the basis for quantitative comparison between the two laws.

Indeed, let $h\in \mathcal{H}_1$. Let $f_h$ be the solution obtained in Theorem~\ref{thm:Stein.solutions} for the Stein equation $\mathcal{A}^{\mathrm{OU}} f_h(X) = h(X) - \EE[h(\mathfrak{Z})]$, and note that $\nabla_X f_h$ is bounded by \eqref{eq:Stein.bound.1} with $m=1$, while the second-order partial derivatives of $f_h$ are bounded by \eqref{eq:Stein.bound.3} with $m=2$ in Theorem~\ref{thm:smoothness.estimates}. In particular, since $n > 3$ implies $\EE[\|\Psi^{-1/2}\mathfrak{T}_n\Sigma^{-1/2}\|_F] < \infty$, we have $f_h\in C_{\mathcal{A}^{\mathrm{L}}}^2(\R^{\nu\times d})$. Then, we have
\[
|\EE[h(\mathfrak{T}_n)] - \EE[h(\mathfrak{Z})]|
= |\EE[\mathcal{A}^{\mathrm{OU}} f_h(\mathfrak{T}_n)]|
= |\EE[\mathcal{A}^{\mathrm{L}} f_h(\mathfrak{T}_n)] - \EE[\mathcal{A}^{\mathrm{OU}} f_h(\mathfrak{T}_n)]|
= |\EE[(\mathcal{A}^{\mathrm{L}} - \mathcal{A}^{\mathrm{OU}}) f_h(\mathfrak{T}_n)]|.
\]
Combining \eqref{eq:OU.process.generator} and \eqref{eq:T.generator} with the change of variable
\[
\mathfrak{Y} = \Psi^{-1/2} \mathfrak{T}_n \Sigma^{-1/2}\sim \mathcal{T}_{\nu\times d} (n,0_{\nu\times d},I_{\nu} \otimes I_d),
\]
we obtain
\[
\begin{aligned}
(\mathcal{A}^{\mathrm{L}} - \mathcal{A}^{\mathrm{OU}}) f_h(\mathfrak{T}_n)
&= \tr\{\Sigma^{1/2} \mathfrak{Y}^{\top} \Psi^{1/2} \nabla_X f_h(\mathfrak{T}_n)\} - \frac{n+\nu+d-1}{n} \, \tr\{\Sigma^{1/2} \mathfrak{Y}^{\top} (I_{\nu} + n^{-1} \mathfrak{Y} \mathfrak{Y}^{\top})^{-1} \Psi^{1/2} \nabla_X f_h(\mathfrak{T}_n)\} \\
&= \frac{1}{n} \tr\{\Sigma^{1/2} \mathfrak{Y}^{\top} \mathfrak{Y} \mathfrak{Y}^{\top} (I_{\nu} + n^{-1} \mathfrak{Y} \mathfrak{Y}^{\top})^{-1} \Psi^{1/2} \nabla_X f_h(\mathfrak{T}_n)\} \\
&\qquad- \frac{\nu+d-1}{n} \, \tr\{\Sigma^{1/2} \mathfrak{Y}^{\top} (I_{\nu} + n^{-1} \mathfrak{Y} \mathfrak{Y}^{\top})^{-1} \Psi^{1/2} \nabla_X f_h(\mathfrak{T}_n)\}.
\end{aligned}
\]
Then, applying the Cauchy--Schwarz inequality to any matrix $A = (a_{ji})_{1\leq j \leq d, 1 \leq i \leq \nu}\in\R^{d\times\nu}$,
\[
\sum_{i=1}^{\nu} \sum_{j=1}^d |a_{ji}| \leq \!\sqrt{\nu d} \!\sqrt{\sum_{i=1}^{\nu} \sum_{j=1}^d |a_{ji}|^2} = \!\sqrt{\nu d} \, \|A\|_F,
\]
yields
\[
\begin{aligned}
|\EE[(\mathcal{A}^{\mathrm{L}} - \mathcal{A}^{\mathrm{OU}}) f_h(\mathfrak{T}_n)]|
&\leq \frac{1}{n} \max_{k\in [\nu], \ell\in [d]} \|\nabla_{k\ell} f_h\|_{\infty} \, \EE\left[\sum_{i=1}^{\nu} \sum_{j=1}^d |(\Sigma^{1/2} \mathfrak{Y}^{\top} \mathfrak{Y} \mathfrak{Y}^{\top} (I_{\nu} + n^{-1} \mathfrak{Y} \mathfrak{Y}^{\top})^{-1} \Psi^{1/2})_{ji}|\right] \\[-1mm]
&\qquad+ \frac{\nu+d-1}{n} \max_{k\in [\nu], \ell\in [d]} \|\nabla_{k\ell} f_h\|_{\infty} \, \EE\left[\sum_{i=1}^{\nu} \sum_{j=1}^d |(\Sigma^{1/2} \mathfrak{Y}^{\top} (I_{\nu} + n^{-1} \mathfrak{Y} \mathfrak{Y}^{\top})^{-1} \Psi^{1/2})_{ji}|\right] \\
&\leq \frac{\!\sqrt{\nu d}}{n} \max_{k\in [\nu], \ell\in [d]} \|\nabla_{k\ell} f_h\|_{\infty} \, \EE\big[\|\Sigma^{1/2} \mathfrak{Y}^{\top} \mathfrak{Y} \mathfrak{Y}^{\top} (I_{\nu} + n^{-1} \mathfrak{Y} \mathfrak{Y}^{\top})^{-1} \Psi^{1/2}\|_F\big] \\
&\qquad+ \frac{(\nu+d-1) \!\sqrt{\nu d}}{n} \max_{k\in [\nu], \ell\in [d]} \|\nabla_{k\ell} f_h\|_{\infty} \, \EE\big[\|\Sigma^{1/2} \mathfrak{Y}^{\top} (I_{\nu} + n^{-1} \mathfrak{Y} \mathfrak{Y}^{\top})^{-1} \Psi^{1/2}\|_F\big].
\end{aligned}
\]
Using \eqref{eq:Stein.bound.1} of Theorem~\ref{thm:smoothness.estimates} to bound $\max_{k\in [\nu], \ell\in [d]} \|\nabla_{k\ell} f_h\|_{\infty} \leq 1$, we deduce
\[
\begin{aligned}
d_{\mathrm{W}}(\tau_n,\gamma)
&= \sup_{h\in\mathcal{H}_1} |\EE[h(\mathfrak{T}_n)] - \EE[h(\mathfrak{Z})]| \\
&\leq \frac{\!\sqrt{\nu d}}{n} \, \EE\big[\|\Sigma^{1/2} \mathfrak{Y}^{\top} \mathfrak{Y} \mathfrak{Y}^{\top} (I_{\nu} + n^{-1} \mathfrak{Y} \mathfrak{Y}^{\top})^{-1} \Psi^{1/2}\|_F\big] + \frac{(\nu+d-1) \!\sqrt{\nu d}}{n} \, \EE\big[\|\Sigma^{1/2} \mathfrak{Y}^{\top} (I_{\nu} + n^{-1} \mathfrak{Y} \mathfrak{Y}^{\top})^{-1} \Psi^{1/2}\|_F\big].
\end{aligned}
\]
Next, by properties of the Frobenius and spectral norms, we have, for any conformable matrices $A,B,C$ with $A$ and $C$ positive definite, that
\begin{equation}\label{eq:Frobenius.spectral.submult}
\|A B C\|_F \leq \|A\|_2 \|B\|_F \|C\|_2,
\end{equation}
see, e.g., \citet[p.~364, 5.6.P20]{MR2978290}. Moreover, since $n^{-1} \mathfrak{Y} \mathfrak{Y}^{\top}\in \mathcal{S}_{+}^{\nu}$ almost surely, its shifted inverse satisfies $\|(I_\nu + n^{-1} \mathfrak{Y} \mathfrak{Y}^{\top})^{-1}\|_2 \leq 1$, so the submultiplicativity of $\|\cdot\|_2$ yields
\[
\|(I_\nu + n^{-1} \mathfrak{Y} \mathfrak{Y}^{\top})^{-1} \Psi^{1/2}\|_2 \leq \|(I_\nu + n^{-1} \mathfrak{Y} \mathfrak{Y}^{\top})^{-1}\|_2 \|\Psi^{1/2}\|_2 \leq \|\Psi^{1/2}\|_2 = \|\Psi\|_2^{1/2}.
\]
Finally, one checks that $\|\mathfrak{Y}^{\top} \mathfrak{Y} \mathfrak{Y}^{\top}\|_F \leq \|\mathfrak{Y}\|_F^3$ and $\|\mathfrak{Y}^{\top}\|_F = \|\mathfrak{Y}\|_F$. Combining these facts and using that $\smash{\Sigma^{1/2}}$ and $\smash{\Psi^{1/2}}$ are positive definite by assumption gives
\[
\begin{aligned}
\|\Sigma^{1/2} \mathfrak{Y}^{\top} \mathfrak{Y} \mathfrak{Y}^{\top} (I_\nu + n^{-1} \mathfrak{Y} \mathfrak{Y}^{\top})^{-1} \Psi^{1/2}\|_F
&\leq \|\Sigma\|_2^{1/2} \|\mathfrak{Y}\|_F^3 \|\Psi\|_2^{1/2}, \\
\|\Sigma^{1/2} \mathfrak{Y}^{\top} (I_\nu + n^{-1} \mathfrak{Y} \mathfrak{Y}^{\top})^{-1} \Psi^{1/2}\|_F
&\leq \|\Sigma\|_2^{1/2} \|\mathfrak{Y}\|_F \|\Psi\|_2^{1/2},
\end{aligned}
\]
and thus, for $n>3$,
\[
\begin{aligned}
d_{\mathrm{W}}(\tau_n,\gamma)
&\leq \frac{\!\sqrt{\nu d}}{n} \, \|\Sigma\|_2^{1/2} \EE[\|\mathfrak{Y}\|_F^3] \, \|\Psi\|_2^{1/2} + \frac{(\nu+d-1) \!\sqrt{\nu d}}{n} \, \|\Sigma\|_2^{1/2} \EE[\|\mathfrak{Y}\|_F] \, \|\Psi\|_2^{1/2} \\
&\leq \frac{\!\sqrt{\nu d}}{n} \big\{\EE[\|\mathfrak{Y}\|_F^3] + (\nu + d - 1) \, \EE[\|\mathfrak{Y}\|_F]\big\} \, \|\Psi\|_2^{1/2} \|\Sigma\|_2^{1/2}.
\end{aligned}
\]

To derive explicit bounds on the Frobenius moments, assume $n > 4$. By the Cauchy--Schwarz inequality and the Frobenius moment formulas in Lemma~\ref{lem:Frobenius.even.moments}, we have
\[
\EE[\|\mathfrak{Y}\|_F^3]
\leq \!\sqrt{\EE[\|\mathfrak{Y}\|_F^2]} \!\sqrt{\EE[\|\mathfrak{Y}\|_F^4]} \leq \!\sqrt{2 \nu d} \!\sqrt{\frac{6 (\nu d)^2}{(1-4/n)}}
= \frac{2\!\sqrt{3}}{\!\sqrt{1 - 4/n}} \, (\nu d)^{3/2}, \qquad
\EE[\|\mathfrak{Y}\|_F]
\leq \!\sqrt{\EE[\|\mathfrak{Y}\|_F^2]} \leq \!\sqrt{2 \nu d},
\]
and thus, for $n>4$,
\begin{equation}\label{tbound0}
d_{\mathrm{W}}(\tau_n,\gamma) \leq \frac{(2\!\sqrt{3} + \!\sqrt{2})}{\!\sqrt{1-4/n}} \, \frac{(\nu d)^2}{n} \, \|\Psi\|_2^{1/2} \|\Sigma\|_2^{1/2}.
\end{equation}

Finally, we deduce that the bound \eqref{tbound} holds for all $n>2$. That the bound \eqref{tbound} holds for $n\geq 5$ follows easily from inequality \eqref{tbound0} since $1/\!\sqrt{1-4/n}\leq \!\sqrt{3}/\!\sqrt{1-2/n}$ for $n\geq 5$, and $\smash{\!\sqrt{3}(2\!\sqrt{3} + \!\sqrt{2}) \leq 10}$. That the bound \eqref{tbound} holds for $n\in (2,5)$ is verified by exploiting a basic property of the Wasserstein distance:
\[
\begin{aligned}
d_{\mathrm{W}}(\tau_n,\gamma)
&\leq \EE[\|\mathfrak{T}_n-\mathfrak{Z}\|_F]
\leq \EE[\|\mathfrak{T}_n\|_F]+ \EE[\|\mathfrak{Z}\|_F]
\leq \!\sqrt{\EE[\|\mathfrak{T}_n\|_F^2]}+\!\sqrt{\EE[\|\mathfrak{Z}\|_F^2]}
= \!\sqrt{\frac{n}{n-2} \, \tr(\Psi) \tr(\Sigma)} + \!\sqrt{\tr(\Psi) \tr(\Sigma)} \\
&\leq \!\sqrt{\frac{n}{n-2} \, \nu \, \|\Psi\|_2 \, d \, \|\Sigma\|_2} + \!\sqrt{\nu \, \|\Psi\|_2 \, d \, \|\Sigma\|_2}
< \frac{2}{\!\sqrt{1-2/n}} \, \!\sqrt{\nu d} \, \|\Psi\|_2^{1/2} \|\Sigma\|_2^{1/2}.
\end{aligned}
\]
Since $(2/\!\sqrt{1-2/n}) \, \!\sqrt{\nu d} \leq (10/\!\sqrt{1-2/n}) \, (\nu d)^2/n$ for $n\in (2,5)$ and $\nu,d\geq 1$, it follows that the bound \eqref{tbound} holds for all $n > 2$. This completes the proof of the bound \eqref{tbound}.

We now prove that the $n^{-1}$ rate is optimal. It suffices to consider the case $\nu=d=1$ and $\Psi=\Sigma=1$. In this case, let $T_n$ denote a random variable following the univariate Student's $t$-distribution with $n$ degrees of freedom. The density of $T_n$ is
\[
p_{T_n}(x) = \frac{\Gamma((n + 1)/2)}{\!\sqrt{\pi n} \, \Gamma(n/2)}\left(1 + \frac{x^2}{n}\right)^{-(n + 1)/2}, \qquad x\in\R.
\]
Also, let $Z$ denote a standard normal random variable, and denote the probability measures of $T_n$ and $Z$ by $\tau_n$ and $\gamma$, respectively. Consider the $1$-Lipschitz test function $h_{\star}:\R\to[0,1]$ defined by $h_{\star}(x)=1 + x$ for $-1<x<0$, by $h_{\star}(x)=1-x$ for $0\leq x<1$, and by $h_{\star}(x)=0$ for $|x|\geq1$. We will prove that
\begin{equation}\label{eq:optimal.to.prove}
d_{\mathrm{W}}(\tau_n,\gamma) \geq |\EE[h_{\star}(T_n)]-\EE[h_{\star}(Z)]| \geq \frac{3\!\sqrt{e}-4}{2\!\sqrt{2e\pi}} \times \frac{1}{n} + \mathcal{O}(n^{-2}), \qquad n\to\infty.
\end{equation}
The first inequality in \eqref{eq:optimal.to.prove} is immediate since $h_{\star}\in\mathcal{H}_1$. We will now establish the second inequality.

By Stirling's approximation, we have that
\[
\frac{\Gamma((n + 1)/2)}{\!\sqrt{\pi n} \, \Gamma(n/2)} = \frac{1}{\!\sqrt{2\pi}}\left(1-\frac{1}{4n} + \mathcal{O}(n^{-2})\right), \qquad n\to\infty.
\]
We also note the elementary expansion, uniformly for $x\in[0,1]$,
\[
\left(1 + \frac{x^2}{n}\right)^{-(n + 1)/2} = e^{-x^2/2}\left[1 + \frac{1}{n}\left(\frac{x^4}{4}-\frac{x^2}{2}\right) + \mathcal{O}(n^{-2})\right], \qquad n\to\infty.
\]
Combining these two expansions yields, uniformly for $x\in[0,1]$,
\[
p_{T_n}(x) = \frac{1}{\!\sqrt{2\pi}}e^{-x^2/2}\left[1 + \frac{1}{n}\left(\frac{x^4}{4}-\frac{x^2}{2}-\frac{1}{4}\right) + \mathcal{O}(n^{-2})\right], \qquad n\to\infty.
\]
Thus, letting $p_Z(x)$ denote the standard normal density and using that the densities $p_{T_n}(x)$ and $p_Z(x)$ are symmetric about the origin, we obtain that
\[
\begin{aligned}
\EE[h_{\star}(T_n)]-\EE[h_{\star}(Z)]
&= 2\int_0^1 (1-x)(p_{T_n}(x)-p_Z(x)) \rd x \\
&= \frac{2}{n}\int_0^1 (1-x) \frac{1}{\!\sqrt{2\pi}}e^{-x^2/2}\left(\frac{x^4}{4}-\frac{x^2}{2}-\frac{1}{4}\right) \rd x + \mathcal{O}(n^{-2})
= -\frac{3\!\sqrt{e}-4}{2\!\sqrt{2e\pi}} \times \frac{1}{n} + \mathcal{O}(n^{-2}),
\end{aligned}
\]
where the integral was evaluated using \emph{Mathematica}. This completes the proof.
\end{proof}

\begin{proof}[\bf Proof of Proposition~\ref{prop:projected.Stein}]\pdfbookmark[2]{Proof of Proposition \ref{prop:projected.Stein}}{proof:projected.Stein}
Let $\smash{\mathfrak{X}^{(1)},\ldots,\mathfrak{X}^{(n)} \stackrel{\mathrm{iid}}{\sim} \mathcal{N}_{\nu\times d}(0_{\nu\times d},\Psi\otimes\Sigma)}$, and let $\mathfrak{X}$ denote a generic copy. We first check that $M_r(\Psi)$ and $N_m(\Sigma)$ are unbiased for $\Sigma$ and $\Psi$, respectively. Fix $r\in \{1,\ldots,R\}$. Writing $\mathfrak{Y}=\Psi^{-1/2}\mathfrak{X}$, we have $\mathfrak{Y}\sim \mathcal{N}_{\nu\times d}(0_{\nu\times d},I_{\nu}\otimes\Sigma)$ and
\[
\mathfrak{X}^{\top}\, \Psi^{-1/2} W_r \Psi^{-1/2}\, \mathfrak{X}
= \mathfrak{Y}^{\top} W_r \mathfrak{Y}.
\]
Using the standard matrix normal identity $\EE[\mathfrak{Y}^{\top} A \mathfrak{Y}]=\tr(A)\,\Sigma$ for $\mathfrak{Y}\sim\mathcal{N}_{\nu\times d}(0_{\nu\times d},I_{\nu}\otimes\Sigma)$ and $A\in\mathcal{S}^{\nu}$ \citep[see, e.g.,][Theorem~2.3.5~(i)]{MR1738933}, we obtain
\[
\EE\big[\mathfrak{X}^{\top}\, \Psi^{-1/2} W_r \Psi^{-1/2}\, \mathfrak{X}\big]
= \EE[\mathfrak{Y}^{\top} W_r \mathfrak{Y}] = \tr(W_r)\,\Sigma.
\]
Since $\mathfrak{X}^{(1)},\ldots,\mathfrak{X}^{(n)}$ are identically distributed, it follows that
\begin{equation}\label{eq:unbiased.M.r}
\EE[M_r(\Psi)]
= \frac{1}{\tr(W_r)} \, \EE\big[(\mathfrak{X}^{(1)})^{\top}\, \Psi^{-1/2} W_r \Psi^{-1/2}\, \mathfrak{X}^{(1)}\big]
= \Sigma.
\end{equation}
The argument for $N_m(\Sigma)$ is analogous: letting $\mathfrak{Z}=\mathfrak{X}\Sigma^{-1/2}$ so that $\mathfrak{Z}\sim \mathcal{N}_{\nu\times d}(0_{\nu\times d},\Psi\otimes I_d)$, one checks that
\begin{equation}\label{eq:unbiased.N.m}
\EE[N_m(\Sigma)]
= \frac{1}{\tr(U_m)} \, \EE\big[\mathfrak{X}^{(1)}\, \Sigma^{-1/2} U_m \Sigma^{-1/2}\, (\mathfrak{X}^{(1)})^{\top}\big]
= \Psi.
\end{equation}

Define the population moment vectors $\smash{\bb{\theta}^{(\Sigma)}\in\R^M}$ and $\smash{\bb{\theta}^{(\Psi)}\in\R^R}$ by
\[
\theta^{(\Sigma)}_m = \EE[y^{(\Sigma)}_m], \qquad
\theta^{(\Psi)}_r = \EE[y^{(\Psi)}_r].
\]
By the unbiasedness just proved,
\[
\begin{aligned}
\theta^{(\Sigma)}_m
&= \EE\left[\frac{1}{R} \sum_{r=1}^R \tr\{M_r(\Psi)\, U_m\}\right]
= \frac{1}{R} \sum_{r=1}^R \tr\{\EE[M_r(\Psi)]\, U_m\}
= \tr(\Sigma U_m), \\
\theta^{(\Psi)}_r
&= \EE\left[\frac{1}{M} \sum_{m=1}^M \tr\{N_m(\Sigma)\, W_r\}\right]
= \frac{1}{M} \sum_{m=1}^M \tr\{\EE[N_m(\Sigma)]\, W_r\}
= \tr(\Psi W_r).
\end{aligned}
\]
Next, use the linear structure of $\Sigma$ and $\Psi$. By assumption, there exist $\bb{\beta}^{\star}\in\R^p$ and $\bb{\alpha}^{\star}\in\R^q$ such that
\[
\Sigma = \Sigma(\bb{\beta}^{\star}) = \sum_{j=1}^p \beta^{\star}_j B_j, \qquad
\Psi = \Psi(\bb{\alpha}^{\star}) = \sum_{\ell=1}^q \alpha^{\star}_{\ell} A_{\ell}.
\]
Therefore,
\[
\theta^{(\Sigma)}_m
= \tr(\Sigma U_m)
= \sum_{j=1}^p \beta^{\star}_j \, \tr(B_j U_m)
= \big(C^{(\Sigma)} \bb{\beta}^{\star}\big)_m, \qquad
\theta^{(\Psi)}_r
= \tr(\Psi W_r)
= \sum_{\ell=1}^q \alpha^{\star}_{\ell} \, \tr(A_{\ell} W_r)
= \big(C^{(\Psi)} \bb{\alpha}^{\star}\big)_r.
\]
Thus, at the population level, we have the linear systems:
\begin{equation}\label{eq:pop.systems}
C^{(\Sigma)} \bb{\beta}^{\star} = \bb{\theta}^{(\Sigma)}, \qquad
C^{(\Psi)} \bb{\alpha}^{\star} = \bb{\theta}^{(\Psi)}.
\end{equation}

Now consider the sample quantities $\smash{\bb{y}^{(\Sigma)} = (y^{(\Sigma)}_1,\ldots,y^{(\Sigma)}_M)^{\top}}$ and $\smash{\bb{y}^{(\Psi)} = (y^{(\Psi)}_1,\ldots,y^{(\Psi)}_R)^{\top}}$, where we recall that
\[
\begin{aligned}
y^{(\Sigma)}_m &= \frac{1}{R} \sum_{r=1}^{R} \tr\{M_r(\Psi)\, U_m\} = \frac{1}{n} \sum_{k=1}^n \frac{1}{R} \sum_{r=1}^R \frac{\tr\{(\mathfrak{X}^{(k)})^{\top}\, \Psi^{-1/2} W_r \Psi^{-1/2}\, \mathfrak{X}^{(k)} U_m\}}{\tr(W_r)}, \\
y^{(\Psi)}_r &= \frac{1}{M} \sum_{m=1}^{M} \tr\{N_m(\Sigma)\, W_r\} = \frac{1}{n} \sum_{k=1}^n \frac{1}{M} \sum_{m=1}^M \frac{\tr\{\mathfrak{X}^{(k)} \, \Sigma^{-1/2} U_m \Sigma^{-1/2}\, (\mathfrak{X}^{(k)})^{\top} W_r\}}{\tr(U_m)}.
\end{aligned}
\]
By combining the strong law of large numbers with \eqref{eq:unbiased.M.r} and \eqref{eq:unbiased.N.m}, we have, as $n\to \infty$,
\begin{equation}\label{eq:y.to.theta}
\bb{y}^{(\Sigma)} \to \bb{\theta}^{(\Sigma)}, \qquad
\bb{y}^{(\Psi)} \to \bb{\theta}^{(\Psi)}, \qquad \text{a.s.}
\end{equation}
Under the full column rank assumptions on $\smash{C^{(\Sigma)}}$ and $\smash{C^{(\Psi)}}$, the least-squares solutions have the closed forms
\[
\widehat{\bb{\beta}} = \left(\big(C^{(\Sigma)}\big)^\top C^{(\Sigma)}\right)^{-1} \big(C^{(\Sigma)}\big)^\top \bb{y}^{(\Sigma)}, \qquad
\widehat{\bb{\alpha}} = \left(\big(C^{(\Psi)}\big)^\top C^{(\Psi)}\right)^{-1} \big(C^{(\Psi)}\big)^\top \bb{y}^{(\Psi)},
\]
so the continuity of the least-squares map, combined with \eqref{eq:pop.systems} and \eqref{eq:y.to.theta} above, yields
\[
\widehat{\bb{\beta}} \to \left(\big(C^{(\Sigma)}\big)^\top C^{(\Sigma)}\right)^{-1} \big(C^{(\Sigma)}\big)^\top \bb{\theta}^{(\Sigma)}
= \bb{\beta}^{\star}, \qquad
\widehat{\bb{\alpha}} \to \left(\big(C^{(\Psi)}\big)^\top C^{(\Psi)}\right)^{-1} \big(C^{(\Psi)}\big)^\top \bb{\theta}^{(\Psi)}
= \bb{\alpha}^{\star}.
\]
Since $\bb{\beta}\mapsto \Sigma(\bb{\beta})$ and $\bb{\alpha}\mapsto \Psi(\bb{\alpha})$ are linear in their arguments, it follows that, as $n\to \infty$,
\[
\widehat{\Sigma} = \Sigma(\widehat{\bb{\beta}}) \to \Sigma(\bb{\beta}^{\star}) = \Sigma, \qquad
\widehat{\Psi} = \Psi(\widehat{\bb{\alpha}}) \to \Psi(\bb{\alpha}^{\star}) = \Psi, \qquad \text{a.s.}
\]
Since $\Sigma\in\mathcal{S}_{++}^d$ and $\Psi\in\mathcal{S}_{++}^{\nu}$, this convergence also implies that $\widehat{\Sigma}\in\mathcal{S}_{++}^d$ and $\widehat{\Psi}\in\mathcal{S}_{++}^{\nu}$ eventually almost surely, although such positive definiteness need not hold at a fixed finite $n$ for the unconstrained least-squares solutions. Finally, $(\widehat{\bb{\beta}},\widehat{\bb{\alpha}})$ minimizes the squared residuals associated with \eqref{eq:structured.Sigma.moments}, so $\smash{\widehat{\Sigma}}$ and $\smash{\widehat{\Psi}}$ solve the projected Stein moment equations in the least-squares sense. This concludes the proof.
\end{proof}

\begin{appendices}

\renewcommand{\thesection}{Appendix~\Alph{section}}

\section{Technical lemmas}\label{app:tech.lemmas}

\renewcommand{\thesection}{\Alph{section}}

The first lemma gives closed-form expressions for expectations of the trace of powers and powers of the trace of inverse Wishart matrices with identity scale.

\begin{lemma}\label{lem:trace.moments.inverse.Wishart}
Let $\mathfrak{W}\sim\mathrm{Wishart}_d(\alpha,I_d)$ with $\alpha>d-1$. If $\alpha > d + 1$, then
\[
\EE[\tr(\mathfrak{W}^{-1})] = \frac{d}{(\alpha-d-1)},
\]
and, if $\alpha > d+3$, then
\[
\EE[\tr(\mathfrak{W}^{-2})] = \frac{(\alpha-d-1) \, d + d^{\hspace{0.2mm}2}}{(\alpha-d)(\alpha-d-1)(\alpha-d-3)}, \qquad
\EE[\tr(\mathfrak{W}^{-1})^2] = \frac{(\alpha-d-2) \, d^{\hspace{0.2mm}2} + 2d}{(\alpha-d)(\alpha-d-1)(\alpha-d-3)}.
\]
\end{lemma}

\begin{proof}[\bf Proof of Lemma~\ref{lem:trace.moments.inverse.Wishart}]
By Theorem~3.1~(i) of \citet{MR968156}, we know that
\[
\EE[\mathfrak{W}^{-1}] = \frac{I_d}{(\alpha-d-1)}, \qquad \alpha > d+1,
\]
and by Corollary~3.1~(i) of \citet{MR968156} (there is a typo there, which we confirmed by Monte Carlo simulations; it should read $\EE[\mathfrak{W}^{-1} \mathfrak{W}^{-1}] = (c_1 + c_2) \Sigma^{-1} \Sigma^{-1} + c_2 \Sigma^{-1} \, \tr(\Sigma^{-1})$), we know that
\[
\EE[\mathfrak{W}^{-2}] = \frac{I_d}{(\alpha-d)(\alpha-d-3)} + \frac{d \, I_d}{(\alpha-d)(\alpha-d-1)(\alpha-d-3)}, \qquad \alpha > d+3.
\]
Taking the trace on both sides of these equations yields the first two claimed formulas. The third formula can be found in Theorem~3.3.18 of \citet{MR1738933} and is due to \citet{Marx1981PhD}. This concludes the proof.
\end{proof}

The second lemma provides expressions for the second and fourth moments of the Frobenius norm of the standard matrix $T$ distribution.

\begin{lemma}\label{lem:Frobenius.even.moments}
Let $\mathfrak{Y}\sim \mathcal{T}_{\nu\times d}(n,0_{\nu\times d},I_{\nu}\otimes I_d)$ for some $\nu,d\in \N$, as defined in \eqref{eq:matrix.T}. Then, for any real $n > 2$,
\[
\EE[\|\mathfrak{Y}\|_F^2] = \frac{\nu d}{(1-2/n)},
\]
and, for any real $n > 4$,
\[
\EE\left[\|\mathfrak{Y}\|_F^4\right] = \frac{2\nu \{(1-2/n) \, d + d^{\hspace{0.2mm}2}/n\} + \nu^2 \{(1-3/n) \, d^{\hspace{0.2mm}2} + 2d/n\}}{(1-1/n)(1-2/n)(1-4/n)}.
\]
In particular, for any real $n > 4$, we have
\[
\EE[\|\mathfrak{Y}\|_F^2] \leq 2 \nu d, \qquad \EE\left[\|\mathfrak{Y}\|_F^4\right] \leq \frac{6 (\nu d)^2}{(1-4/n)}.
\]
\end{lemma}

\begin{proof}[\bf Proof of Lemma~\ref{lem:Frobenius.even.moments}]
By Theorem~4.2.1 of \citet{MR1738933} (recall that we have an extra $\!\sqrt{n}$-scaling in our definition), note that if $\mathfrak{Z}\sim \mathcal{N}_{\nu\times d}(0_{\nu\times d},I_{\nu} \otimes I_d)$ and $\mathfrak{W}\sim \mathrm{Wishart}_d(n+d-1,I_d)$ are independent, then $\mathfrak{Y} \smash{\stackrel{\mathrm{law}}{=}} \!\sqrt{n} \, \mathfrak{Z} \mathfrak{W}^{-1/2}$. Therefore, by the invariance of the trace under cyclic permutations, we have
\[
n^{-1} \, \EE[\|\mathfrak{Y}\|_F^2]
= \EE\big[\tr((\mathfrak{Z} \mathfrak{W}^{-1/2})^{\top} \mathfrak{Z} \mathfrak{W}^{-1/2})\big]
= \EE\big[\tr(\mathfrak{W}^{-1} \mathfrak{Z}^{\top} \mathfrak{Z})\big].
\]
Upon conditioning on $\mathfrak{W}$, using the fact that $I_d$ is the column-scale matrix of $\mathfrak{Z}$, and applying Lemma~\ref{lem:trace.moments.inverse.Wishart} with $\alpha = n + d - 1$ and $n > 2$, the above is equal to
\[
\EE\big[\tr\{\mathfrak{W}^{-1} \EE[\mathfrak{Z}^{\top} \mathfrak{Z} \mid \mathfrak{W}]\}\big]
= \EE\left[\tr\left\{\mathfrak{W}^{-1} \sum_{i=1}^{\nu} \EE[\mathfrak{Z}_{i\bigcdot}^{\top} \mathfrak{Z}_{i\bigcdot}]\right\}\right]
= \EE\left[\tr(\mathfrak{W}^{-1} \nu \, I_d)\right]
= \nu \, \EE[\tr(\mathfrak{W}^{-1})]
= \frac{\nu d}{(n-2)},
\]
so that
\[
\EE[\|\mathfrak{Y}\|_F^2] = \frac{\nu d}{(1-2/n)},
\]
proving the first claim of the lemma.

Next, letting $X_i = \mathfrak{Z}_{i\bigcdot} \mathfrak{W}^{-1} \mathfrak{Z}_{i\bigcdot}^{\top}$ for $i\in [\nu]$, we have
\begin{equation}\label{eq:fourth.power.Frobenius}
\begin{aligned}
n^{-2} \, \EE\left[\|\mathfrak{Y}\|_F^4\right]
&= \EE\left[\left(\tr(\mathfrak{W}^{-1} \mathfrak{Z}^{\top} \mathfrak{Z})\right)^2\right]
= \EE\left[\EE\left[\left(\sum_{i=1}^{\nu} \mathfrak{Z}_{i\bigcdot} \mathfrak{W}^{-1} \mathfrak{Z}_{i\bigcdot}^{\top}\right)^2 \,\bigm|\,\mathfrak{W}\right]\right]
= \EE\left[\sum_{i=1}^{\nu} \EE[X_i^2\mid \mathfrak{W}] + 2 \sum_{1\leq i < k \leq \nu}\EE[X_i X_k\mid \mathfrak{W}]\right] \\
&= \EE\left[\sum_{i=1}^{\nu} \{\Var[X_i\mid \mathfrak{W}] + (\EE[X_i\mid \mathfrak{W}])^2\} + 2\sum_{1\leq i < k \leq \nu}\EE[X_i X_k\mid \mathfrak{W}]\right].
\end{aligned}
\end{equation}
For any multivariate normal random vector $\bb{Z}\sim \mathcal{N}_d(\bb{0}_d,V)$ with $A,V\in \mathcal{S}_{++}^d$, note that
\[
\EE[\bb{Z}^{\top} A \bb{Z}]
= \EE[\tr(\bb{Z}^{\top} A \bb{Z})] = \tr(\EE[A \bb{Z} \bb{Z}^{\top}]) = \tr(A \EE[\bb{Z} \bb{Z}^{\top}]) = \tr(A V),
\]
\[
\begin{aligned}
\EE[(\bb{Z}^{\top} A \bb{Z})^2]
&= \sum_{i,j,k,\ell=1}^d \hspace{-2mm} A_{ij} A_{k\ell} \EE[Z_i Z_j Z_k Z_{\ell}]
= \sum_{i,j,k,\ell=1}^d \hspace{-2mm} A_{ij} A_{k\ell} \{ V_{ij} V_{k\ell} + V_{ik} V_{j\ell} + V_{i\ell} V_{jk} \} = \tr(A V)^2 + 2 \, \tr(A V A V),
\end{aligned}
\]
and
\[
\Var(\bb{Z}^{\top} A \bb{Z})
= \EE[(\bb{Z}^{\top} A \bb{Z})^2] - \EE[\bb{Z}^{\top} A \bb{Z}]^2 = 2 \, \tr(A V A V).
\]
Hence, for all $i,k\in [\nu]$ such that $i\neq k$, it follows that
\[
\begin{aligned}
\Var(X_i\mid \{\mathfrak{W} = W\})
&= 2 \, \tr\{W^{-1} \Var(\mathfrak{Z}_{i\bigcdot}) W^{-1} \Var(\mathfrak{Z}_{i\bigcdot})\}
= 2 \, \tr\{W^{-1} I_d W^{-1} I_d\}
= 2 \, \tr(W^{-2}), \\
\EE[X_i\mid \{\mathfrak{W} = W\}]
&= \tr\{W^{-1} \Var(\mathfrak{Z}_{i\bigcdot})\} = \tr(W^{-1} I_d) = \tr(W^{-1}), \\
\EE[X_i X_k\mid \{\mathfrak{W} = W\}]
&= \EE[X_i \mid \{\mathfrak{W} = W\}] \EE[X_k\mid \{\mathfrak{W} = W\}] = \tr(W^{-1})^2.
\end{aligned}
\]
Substituting these expressions into \eqref{eq:fourth.power.Frobenius} above and applying Lemma~\ref{lem:trace.moments.inverse.Wishart} with $\alpha = n + d - 1$ and $n > 4$, we find
\[
\begin{aligned}
n^{-2} \, \EE\left[\|\mathfrak{Y}\|_F^4\right]
&= \EE\left[\nu \left\{2 \, \tr(\mathfrak{W}^{-2}) + \tr(\mathfrak{W}^{-1})^2\right\} + 2 \binom{\nu}{2} \tr(\mathfrak{W}^{-1})^2\right]
= 2\nu \, \EE[\tr(\mathfrak{W}^{-2})] + \nu^2 \, \EE[\tr(\mathfrak{W}^{-1})^2] \\
&= 2\nu \, \frac{\{(n-2) \, d + d^{\hspace{0.2mm}2}\}}{(n-1)(n-2)(n-4)} + \nu^2 \, \frac{\{(n-3) \, d^{\hspace{0.2mm}2} + 2d\}}{(n-1)(n-2)(n-4)},
\end{aligned}
\]
so that
\[
\EE\left[\|\mathfrak{Y}\|_F^4\right] = \frac{2\nu \{(1-2/n) \, d + d^{\hspace{0.2mm}2}/n\} + \nu^2 \{(1-3/n) \, d^{\hspace{0.2mm}2} + 2d/n\}}{(1-1/n)(1-2/n)(1-4/n)} \leq \frac{6 (\nu d)^2}{(1-4/n)},
\]
where the inequality is derived by elementary means by exploiting the fact that $n>4$ and $\nu,d\geq 1$; we omit the tedious details. This completes the proof of the lemma.
\end{proof}

The third lemma gives an exponential contraction for $\alpha$-H\"older continuous test functions for the matrix Ornstein--Uhlenbeck semigroup.

\begin{lemma}[$\alpha$-H\"older contraction]\label{lem:Walpha.contraction.OU}
Let $\alpha\in(0,1]$. Let $(\mathfrak{X}_t)_{t\geq 0}$ be the matrix Ornstein--Uhlenbeck process defined in \eqref{eq:OU.process}, with transition semigroup $(\mathcal{P}^{\mathrm{OU}}_t)_{t\geq 0}$ and stationary limiting distribution $\gamma = \mathcal{N}_{\nu\times d}(0_{\nu\times d},\Psi\otimes\Sigma)$. Recall the notation used in Proposition~\ref{prop:OU.process.distribution} and Theorem~\ref{thm:Stein.solutions}. Then, for all probability measures $\mu,\eta$ on $\R^{\nu\times d}$ with finite $\alpha$-moments and all $t\geq 0$,
\begin{equation}\label{eq:OU.Holder.contraction}
d_{\mathrm{HK},\alpha}\big(\mu \, \mathcal{P}^{\mathrm{OU}}_t, \eta \, \mathcal{P}^{\mathrm{OU}}_t\big)\leq e^{-\alpha t} d_{\mathrm{HK},\alpha}(\mu,\eta).
\end{equation}
In particular, since $\gamma \, \mathcal{P}^{\mathrm{OU}}_t = \gamma$, we have, for every $X\in \R^{\nu\times d}$,
\begin{equation}\label{eq:OU.Holder.delta}
d_{\mathrm{HK},\alpha}(\delta_X \mathcal{P}^{\mathrm{OU}}_t, \gamma)\leq e^{-\alpha t} d_{\mathrm{HK},\alpha}\big(\delta_X, \gamma\big).
\end{equation}
\end{lemma}

\begin{proof}[\bf Proof of Lemma~\ref{lem:Walpha.contraction.OU}]
Fix two initial starting points $X,Y\in \R^{\nu\times d}$, and consider two corresponding solutions of the SDE \eqref{eq:OU.process} driven by the same Brownian matrix $(\mathfrak{B}_t)_{t\geq 0}$:
\[
\begin{aligned}
\rd \mathfrak{X}^X_t &= -\mathfrak{X}^X_t \rd t+\!\sqrt{2} \Psi^{1/2} \rd \mathfrak{B}_t \Sigma^{1/2}, \qquad \mathfrak{X}^X_0=X, \\
\rd \mathfrak{X}^Y_t &= -\mathfrak{X}^Y_t \rd t+\!\sqrt{2} \Psi^{1/2} \rd \mathfrak{B}_t \Sigma^{1/2}, \qquad \mathfrak{X}^Y_0=Y.
\end{aligned}
\]
The difference process $\Delta_t=\mathfrak{X}^X_t-\mathfrak{X}^Y_t$ solves $\rd \Delta_t=-\Delta_t \rd t$ with $\Delta_0=X-Y$, and hence
\begin{equation}\label{eq:OU.synchronous.diff.alpha}
\Delta_t=e^{-t}(X-Y) \qquad \text{for all } t\geq 0.
\end{equation}
Let $h\in C^{0,\alpha}(\R^{\nu\times d})$. Using \eqref{eq:OU.synchronous.diff.alpha}, we have
\[
|(\mathcal{P}^{\mathrm{OU}}_t h)(X)-(\mathcal{P}^{\mathrm{OU}}_t h)(Y)|
= \big|\EE[h(\mathfrak{X}^X_t)-h(\mathfrak{X}^Y_t)]\big|
\leq \EE\big[|h(\mathfrak{X}^X_t)-h(\mathfrak{X}^Y_t)|\big]
\leq [h]_{\alpha}\, \EE[\|\Delta_t\|_F^{\alpha}]
= [h]_{\alpha}\, e^{-\alpha t} \|X-Y\|_F^{\alpha},
\]
and thus
\[
[\mathcal{P}^{\mathrm{OU}}_t h]_{\alpha}\leq e^{-\alpha t}[h]_{\alpha}.
\]
Therefore, we obtain, for arbitrary probability measures $\mu,\eta$ on $\R^{\nu\times d}$ with finite $\alpha$-moments,
\[
\begin{aligned}
d_{\mathrm{HK},\alpha}(\mu \, \mathcal{P}^{\mathrm{OU}}_t, \eta \, \mathcal{P}^{\mathrm{OU}}_t)
&= \sup_{[h]_{\alpha}\leq 1} \Big|\int_{\R^{\nu\times d}} h(Y) \, (\mu \, \mathcal{P}^{\mathrm{OU}}_t)(\rd Y) - \int_{\R^{\nu\times d}} h(Y) \, (\eta \, \mathcal{P}^{\mathrm{OU}}_t)(\rd Y)\Big| \\
&= \sup_{[h]_{\alpha}\leq 1} \Big|\int_{\R^{\nu\times d}} \int_{\R^{\nu\times d}} h(Y) P_t(X,\rd Y) \, \mu(\rd X) - \int_{\R^{\nu\times d}} \int_{\R^{\nu\times d}} h(Y) P_t(X,\rd Y) \, \eta(\rd X)\Big| \\
&= \sup_{[h]_{\alpha}\leq 1} \Big|\int_{\R^{\nu\times d}} \mathcal{P}^{\mathrm{OU}}_t h(X) \, \mu(\rd X) - \int_{\R^{\nu\times d}} \mathcal{P}^{\mathrm{OU}}_t h(X) \, \eta(\rd X)\Big| \\
&\leq \sup_{[h]_{\alpha}\leq 1} [\mathcal{P}^{\mathrm{OU}}_t h]_{\alpha} \, d_{\mathrm{HK},\alpha}(\mu,\eta) \\
&\leq e^{-\alpha t} \, d_{\mathrm{HK},\alpha}(\mu,\eta).
\end{aligned}
\]
This is \eqref{eq:OU.Holder.contraction}. Taking $\eta=\gamma$ and using the invariance $\gamma \, \mathcal{P}^{\mathrm{OU}}_t=\gamma$ from \eqref{eq:invariance} yields \eqref{eq:OU.Holder.delta}.
\end{proof}

\end{appendices}

\section*{Funding}
\addcontentsline{toc}{section}{Funding}

R.\ E.\ Gaunt is supported by EPSRC grants EP/Y008650/1 and UKRI068. F.\ Ouimet is supported by the Natural Sciences and Engineering Research Council of Canada (NSERC) through Discovery Grant RGPIN-2026-04471 and Discovery Launch Supplement DGECR-2026-00449.

\section*{Acknowledgments}
\addcontentsline{toc}{section}{Acknowledgments}

We thank the referees for their careful reading of the manuscript and for their constructive comments, which helped us improve the presentation and clarity of the paper.

\section*{References}
\addcontentsline{toc}{section}{References}

\bibliographystyle{myjmva}
\bibliography{bib_JMVA}

\end{document}